\documentclass[10pt]{article}
\usepackage{amsfonts}
\usepackage{amssymb,amsmath,amsthm, amsfonts}

\def\sqr#1#2{{\vcenter{\vbox{\hrule height.#2pt
              \hbox{\vrule width.#2pt height#1pt \kern#1pt \vrule width.#2pt}
          \hrule height.#2pt}}}}
\def\sqr#1#2{{\vcenter{\vbox{\hrule height.#2pt
              \hbox{\vrule width.#2pt height#1pt \kern#1pt \vrule width.#2pt}
              \hrule height.#2pt}}}}
\def\3n{\negthinspace \negthinspace \negthinspace }
\def\2n{\negthinspace \negthinspace }
\def\1n{\negthinspace }

\def\={\buildrel \triangle \over =}

\def\esssup{\mathop{\rm esssup}}
\def\essinf{\mathop{\rm essinf}}

\def\sup{\mathop{\rm sup}}
\def\inf{\mathop{\rm inf}}

\def\inf{\hbox{\rm inf$\,$}}
\def\esssup{\hbox{\rm esssup}}
\def\essinf{\hbox{\rm essinf}}
\def\sup{\hbox{\rm sup}}
\def\inf{\hbox{\rm inf}}

\def\|{\Big |}
\def\({\Big (}
\def\){\Big )}
\def\[{\Big[}
\def\]{\Big]}
\def\be{\begin{equation}}
\def\bel{\begin{equation}\label}
\def\ee{\end{equation}}
\def\bt{\begin{theorem}}
\def\bcd{\begin{condition}}
\def\ecd{\end{condition}}
\def\et{\end{theorem}}
\def\bc{\begin{corollary}}
\def\ec{\end{corollary}}
\def\bde{\begin{definition}}
\def\ede{\end{definition}}
\def\bl{\begin{lemma}}
\def\el{\end{lemma}}
\def\bp{\begin{proposition}}
\def\ep{\end{proposition}}
\def\bex{\begin{example}}
\def\eex{\end{example}}
\def\br{\begin{remark}}
\def\er{\end{remark}}
\def\ba{\begin{array}}
\def\ea{\end{array}}
\def\ed{\end{document}}

\def\square#1{\vbox{\hrule\hbox{\vrule height#1%
     \kern#1\vrule}\hrule}}
\def\rectangle#1#2{\vbox{\hrule\hbox{\vrule height#1%
     \kern#2\vrule}\hrule}}

\font\tenbb=msbm10 \font\sevenbb=msbm7 \font\fivebb=msbm5

\newfam\bbfam
\scriptscriptfont\bbfam=\fivebb \textfont\bbfam=\tenbb
\scriptfont\bbfam=\sevenbb

\newtheorem{lemma}{Lemma}[section]
\newtheorem{remark}{Remark}[section]
\newtheorem{example}{Example}[section]
\newtheorem{theorem}{Theorem}[section]
\newtheorem{corollary}{Corollary}[section]

\newtheorem{definition}{Definition}[section]
\newtheorem{proposition}{Proposition}[section]
\newtheorem{condition}{Condition}[section]

\makeatletter
   
   \@addtoreset{equation}{section}
\makeatother

\begin{document}

\title{Value Function of Differential Games without Isaacs Conditions. An Approach with Non-Anticipative Mixed Strategies}
\author{ Rainer Buckdahn$^{1,3}$;  Juan Li$^{2,3}${\footnote{Juan Li is the corresponding author.}} ; Marc Quincampoix$^{1}$\\
{\small $^{1}$ Laboratoire de  Math\'{e}matiques,  CNRS-UMR 6205,
Universit\'{e} de
Bretagne Occidentale,}\\
 {\small 6, avenue Victor Le Gorgeu, CS 93837, 29238 Brest
cedex 3, France.}\\
{\small $^{2}$ School of Mathematics and Statistics, Shandong University at Weihai,
Weihai 264209, P.R. China;}\\
{\small $^{3}$ Institute for Advanced Study, Shandong University, Jinan 250100,
P.R. China}\\
{\small {\it E-mails: rainer.buckdahn@univ-brest.fr;
marc.quincampoix@univ-brest.fr; juanli@sdu.edu.cn.}}\\
\date{December 23, 2011 }}\maketitle
\noindent{\bf Abstract}\hskip4mm In the present paper we investigate the problem
of the existence of a value for differential games without Isaacs
condition. For this we introduce a suitable concept of mixed
strategies along a partition of the time interval, which are
associated with classical nonanticipative strategies (with delay).
Imposing on the underlying controls for both players a conditional
independence property, we obtain the existence of the value in mixed
strategies as the limit of the lower as well as of the upper value
functions along a sequence of partitions which mesh tends to zero.
Moreover, we characterize this value in mixed strategies as the
unique viscosity solution of the corresponding
Hamilton-Jacobi-Isaacs equation.

\section{\large{Introduction}}

In the present work we consider 2-person zero-sum differential games
which dynamics is defined through the doubly controlled differential
equation

\be\label{I.1}\frac{d}{ds}X_s =f(s, X_s, u_s, v_s), \; s\in [t,T],\ee
and which pay-off functional is described by

\vskip -0.5cm \be\label{I.2} J:= g(X_T).\ee \vskip -0.1cm

\noindent The initial data $(t,x)$ are in $[0,T]\times R^d$. Given
two compact metric control state spaces $U$ and $V$, the both
players use control processes $u=(u_s)$ and $v=(v_s)$ with values in
$U$ and $V$, respectively. They control the state space process
$X=(X_s)$ which takes its values in $R^d$; its dynamics is driven by
a bounded, continuous function $f=(f(t,x,u,v)):[0,T]\times R^d\times
U\times V\rightarrow R^d$ which is Lipschitz in $x$, uniformly with
respect to $(u,v)$, and the terminal pay-off function $g : R ^d \to
R $ is supposed to be bounded and Lipschitz. Under these assumptions
on $f$ the above equation has a unique solution
$X=(X_s)_{s\in[t,T]}$, denoted by $X^{t,xu,v}$ in order to indicate
the dependence on the initial data $(t,x)$ and the control processes
$u=(u_s)$ and $v=(v_s)$ chosen by player 1 and 2, respectively; and
for the associated pay-off functional we write $J(t,x;u,v)$. While
the objective of the first player consists in maximizing the pay-off
at terminal time $T$, the second player's objective is to minimize
it.

One important issue in the  theory of 2-person zero-sum differential
games is the study of conditions under which the {\it value} of the
game exists, i.e., under which the lower and the upper value
functions of the game coincide. Indeed, with an appropriate concept
of strategies, which will be introduced in Section 2, two value
functions can be introduced, the lower and the upper one. For the
case of a deterministic differential game with dynamics (\ref{I.1})
and pay-off (\ref{I.2}) the lower value function $V:[0,T]\times R^d$
and the upper one $U:[0,T]\times R^d$ are defined as follows:
\begin{equation}V(t,x) = \sup _ \alpha \inf _\beta
J(t,x, \alpha, \beta), \ \ \ U (t,x) = \inf _\beta \sup _ \alpha
J(t,x, \alpha, \beta),\ \ (t,x)\in [0,T]\times R^d ,\end{equation}
where $ \alpha $ runs the set of admissible strategies for the first
player, and $ \beta$ those for the second one. Given such a couple
of admissible strategies $(\alpha,\beta)$, we define the associated
pay-off functional $J(t,x, \alpha , \beta )$ through the unique
couple of controls $(u,v)$ such that $\alpha(v)=u$ and $\beta(u)=v:$
$J(t,x, \alpha , \beta ):=J(t,x,u,v).$

In the literature, since the pioneering works of Isaacs, there have
been many works showing the existence of the value of the
game, this means the equality between the lower and the upper value
functions, under the so-called Isaacs condition saying that, for all
$(t,x,p) \in [0,T] \times R^d \times R ^d$,
\begin{equation}\label{isaacs} \displaystyle{
\sup _{u \in U } \inf _{v \in V } f(t,x,u,v)p = \inf _{v \in V }\sup
_{u \in U } f(t,x,u,v)p.}
\end{equation}
Moreover, under this condition (\ref{isaacs}) the value function
$V(=U)$ solves a partial differential equation, the so-called
Hamilton-Jacobi-Isaacs equation. Such an existence result for the
value was obtained in \cite{ES} in the context of nonanticipative
Varaiya-Roxin-Elliot-Kalton strategies, see \cite{ELKA72}, \cite{
ROX69} and \cite{VL67}, and also in \cite{BCQ06}, \cite{CQS3} and
\cite{PQ2000}, but here for differential games with constraints. As
concerns the context of positional strategies, we refer to
\cite{KRSU} for similar results.

For 2-person zero-sum stochastic differential games the existence of
a value was obtained in \cite{FS} and later revisited and
generalized in \cite{Buckdahn-Li-2008}. We also refer the reader to
\cite{BCQ11} and the references therein for an overview and a more
complete description of these approaches.

\medskip

Our main goal in the present paper is to investigate the problem of
the existence of a value without Isaacs condition. Having other
approaches in the classical theory of differential games in mind, it
is not surprising that we need a proper, suitable notion of mixed
strategies. This proper notion of mixed strategies related with a
suitable randomization  allows to show that the lower and the upper
value functions {\it defined in mixed strategies} coincide.
Moreover, we prove that the value in mixed strategies
$V=(V(t,x)=U(t,x))$ solves in viscosity sense the
Hamilton-Jacobi-Isaacs equation \be\label{I.3}\begin{array}{rlll}
\displaystyle \frac{\partial}{\partial t}V(t,x)+
H(t,V(t,x),\nabla_x V(t,x)) &=& 0, \ &(t,x)\in [0,T]\times R^d,\\
V(T,x)&=&g(x), \ &x\in R^d,
\end{array}\ee
which Hamiltonian is given by
\begin{equation}\label{Hmixte}  \displaystyle{H(t,x,p) : =
\inf_{\nu \in \Delta V} \sup _{ \mu \in \Delta V }\int_{V}\int_{U}
f(t,x,u,v)\mu(du)\nu(dv)p,\ \ (t,x,p) \in [0,T] \times R^d \times R
^d .} \end{equation} Here $ \Delta U$ and $\Delta V$ denote the set
of probability measures on the set $U$ and $V$ (equipped with the
Borel $\sigma $-field), respectively. It is worth pointing out that
the supremum and the infimum in (\ref{Hmixte}) commute due to the
classical minmax theorem. This commutation between the supremum and
the infimum in (\ref{Hmixte}) constitutes also the key in the proof
of the existence of the value in mixed strategies; it can be
regarded as an automatically satisfied Isaacs condition concerning
$\Delta U$ and $\Delta V$ interpreted as control state spaces.
Having this in mind one could immediately define mixed strategies as
nonanticipative strategies with delay for controls taking their
values in $ \Delta U$ and $\Delta V$, respectively. This would lead
to the same value of the game, given by (\ref{I.3}).

But proceeding like that would mean to use relaxed controls.
However, being interested in strong controls, i.e., controls taking
their values in the given control state space $U$ and $V$,
respectively, we define controls and strategies, where the
randomness--necessary for defining the concept of mixed strategies--appears in the choices of the players and not in the values of the
controls. In this sense our work can be considered as an extension
of the famous Kuhn Theorem for repeated games  ( cf \cite{Kuhn} and also \cite{Aumann}) to the context of
deterministic differential games.

To the best of our knowledge, the existence of the value for
differential games without Isaacs condition was only investigated in
the case of positional strategies in \cite{KRSU}, but with different
techniques. Moreover, the nonanticipative strategies used in
\cite{BCQ06, CQS3, PQ2000} do not allow to write the game in a {\it
normal form} (i.e., to play a strategy of one player against a
strategy of the other one) and, consequently, they are not
appropriate for the definition of mixed strategies. Here in our work
we use the concept of nonanticipativity with delays (see
\cite{BCR04, BCQ11} and \cite{CaQu}) and we define a corresponding
notion of mixed strategies.

\medskip

Let us explain the organization of the paper and link it with some
explanation concerning our approach: {\it Section 2} is devoted to
some preliminaries. We introduce there, in particular, the
underlying filtered probability space $(\Omega,{\cal
F},\mathbb{F}=({\cal F}_j)_{j\ge 1}, P)$ which we use for the
randomization of the controls and the strategies. Given an arbitrary
partition $\Pi$ of the interval $[0,T]$, we introduce the admissible
controls for both players along this partition $\Pi$ and the
corresponding nonanticipative strategies with delay (for short
NAD-strategies). The specificity of the choice of our admissible
controls along the partition $\Pi=\{0=t_0<\dots <t_n=T\}$ consists
in the fact that, given the available information ${\cal F}_i$ at
time $t_i$, the admissible control processes for player 1 restricted
to the time interval $[t_i,t_{i+1})$ are independent of those for
player 2. This conditional independence of the control processes on
subintervals defined by the partition $\Pi$ turns out to be the
crucial element in our approach. We show that, along the partition
$\Pi$, for every couple of NAD strategies $\alpha,\, \beta,$ there
exists a unique couple of admissible controls $u,v$ of player 1 and
2, respectively, such that $\alpha(v)=u$ and $\beta(u)=v$. This
allows to give a sense to the pay-off functional
$J(t,x;\alpha,\beta)$. Since the admissible controls are random,
also the pay-off functionals are random, and so are, a priori,
$V^\Pi$ and $U^\Pi$, the lower and the upper value functions along
the partition $\Pi$. In {\it Section 3} we show that $V^\Pi$ and
$U^\Pi$ satisfy along the partition $\Pi$ the dynamic programming
principle. This principle allows to prove with the help of a
backward iteration that $V^\Pi$ and $U^\Pi$ are deterministic. For
this a key result is that $V^\Pi$ and $U^\Pi$ are invariant with
respect to a certain class of bijective transformations
$\tau:\Omega\rightarrow \Omega$ which law is equivalent to the
underlying probability measure $P$, combined with a statement saying
that any random variable with such an invariance property has to
coincide $P$-almost surely with a constant. The proof extends an
idea coming from \cite{Buckdahn-Li-2008}, where it was developed for
a Brownian framework. Furthermore, the fact that $V^\Pi$ and $U^\Pi$
are deterministic, allows to prove that
\begin{equation}V^\Pi(t,x) = \sup _ \alpha \inf _\beta
E[J(t,x, \alpha, \beta)], \ \ \ U^\Pi(t,x) = \inf _\beta \sup _
\alpha E[J(t,x, \alpha, \beta)],\ \ (t,x)\in [0,T]\times R^d
,\end{equation} where  $ \alpha $ runs the set of NAD-strategies
along $\Pi$ for the first player, and $ \beta$ those for the second
player. This combined with standard estimates yields that $V^\Pi$
and $U^\Pi$ are jointly Lipschitz in $(t,x)$, with a Lipschitz
constant which does not depend on the partition $\Pi$. From there we
deduce in {\it Section 4} that the lower and the upper value
functions $V^\Pi$ and $U^\Pi$ converge uniformly on compacts to the
unique solution of the Hamilton-Jacobi-Isaacs equation (\ref{I.3}),
as the maximal distance $|\Pi|$ between two neighbouring points of
the partition $\Pi$ tends to zero. Consequently, the limits of
$V^\Pi$ and $U^\Pi$, $V:=\lim_{|\Pi|\rightarrow 0}V^\Pi$ and
$U:=\lim_{|\Pi|\rightarrow 0}U^\Pi$ exist and coincide: $V=U$ is the
value in mixed strategies of the game.

\section{\large{Preliminaries}}

Let $\lambda_2(dx)=dx$ denote the two-dimensional Borel measure
defined on the quadrate $[0,1]^2\subset R^2$ endowed with the Borel
field ${\cal B}([0,1]^2)$. Denoting by $\mathbb{N}$ the set of all
positive integers we introduce our underlying probability space
$(\Omega,{\cal F},P)$ as product space

\centerline{$\displaystyle (\Omega,{\cal
F},P):=\left(([0,1]^2)^{\mathbb{N}},{\cal B} ([0,1]^2
)^{\otimes\mathbb{N}}, \lambda_2^{\otimes\mathbb{N}}\right),$}

\noindent i.e., $\Omega=\{\omega=(\omega_j)_{j\ge 1}\, |\,
\omega_j\in[0,1]^2,\, j\ge 1\}$ is the space of all $[0,1]^2$-valued
sequences, endowed with the product Borel-field ${\cal F}={\cal
B}([0,1]^2 )^{\otimes\mathbb{N}}$ and the product probability
measure $P= \lambda_2^{ \otimes\mathbb{N}}$. Moreover, letting
$\zeta_j=(\zeta_{j,1},\zeta_{j,2}):\Omega \longrightarrow [0,1]^2$
denote the coordinate mapping on $\Omega:$

\centerline{$\zeta_j(\omega)=(\zeta_{j,1}(\omega),\zeta_{j,2}(\omega))=
(\omega_{j,1},\omega_{j,2}),\ \
\omega=((\omega_{j,1},\omega_{j,2}))_{j\ge 1} \in\Omega,$}

\noindent we have that ${\cal F}$ is the smallest $\sigma$-field
over $\Omega$, with respect to which all coordinate mappings
$\zeta_j,\, j\ge 1,$ are measurable. In what follows we will also
need the $\sigma$-fields ${\cal G}_j:= \zeta_{j,1}^{-1}({\cal
B}([0,1]))=\{\{\zeta_{j,1}\in\Gamma\}\, |\, \Gamma \in{\cal
B}([0,1])\}$ and ${\cal H}_j:=\zeta_{j,2}^{-1}({\cal B}([0,1]))$
generated by $\zeta_{j,1}$ and $\zeta_{j,2}$, respectively, $j\ge
1$, as well as the $\sigma$-field

\centerline{$\displaystyle {\cal F}_j:=\sigma\left\{\cup_{i\le j}({\cal G}_i\cup
{\cal H}_i)\right\}=\sigma\{\zeta_i,\, 1\le i\le j\},$}

\noindent generated by the coordinate mappings
$\zeta_1,\dots,\zeta_j$, for $j\ge 1.$ We remark that, for all $j\ge
1,$ the $\sigma$-fields  ${\cal G}_j, {\cal H}_j$ and ${\cal
F}_{j-1}$ are independent. Moreover, $\mathbb{F}=({\cal F}_j)_{j\ge
1}$ forms a time-discrete filtration, and ${\cal F}=\vee_{j\ge
1}{\cal F}_j\, (:=\sigma\{\cup_{j\ge 1}{\cal F}_j\}\, )$. We also
recall that a random time $\tau:\Omega\rightarrow \{0,1,2,\dots\}$
is an $\mathbb{F}$-stopping time, if $\{\tau=j\}\in{\cal F}_j,\,
j\ge 0.$

Let $U$ and $V$ be compact metric spaces; by $\Delta U$ and $\Delta
V$ we denote the space of probability measures on $(U,{\cal B}(U))$
and on $(V,{\cal B}(V)),$ respectively. The fact that all
probability measure $\mu\in \Delta U$ ($\nu\in \Delta V$, resp.)
coincides with the law of a suitable $U$-valued random variable
($V$-valued random variable, resp.) defined over the space
$([0,T],{\cal B}([0,T]))$ endowed with the one-dimensional Borel
measure (it's an elementary consequence of Skorohod's Representation
Theorem, refer to pp 70 in \cite{B99}), implies, in particular, that

\centerline{$\displaystyle \Delta U=\{P_\xi\, |\, \xi\in
L^0(\Omega,{\cal G}_j,P;U)\}$\footnote{As usual,
$L^0(\Omega,{\cal G}_j,P;U)\}$ denotes the space of all $U$-valued
random variables defined on $ (\Omega,{\cal G}_j,P).$},\ $\Delta
V=\{P_\xi\, |\, \xi\in L^0(\Omega,{\cal H}_j,P;U)\}, \, j\ge 1.$}

\smallskip

In order to introduce the dynamics of the controlled system we want
to investigate, we shall begin with defining the admissible controls
for the both players. We define them along a partition  $\Pi$ of the
time interval $[0,T]$.

\begin{definition}\label{adm.control}(admissible control) A process
$u\in L_{\cal F}^0(0,T;U)$\footnote{$L_{\cal
F}^0(0,T;U)$ denotes the space of all measurable $U$-valued
processus $u=(u_t)_{t\in[0,T]}$ such that $u_t$ is ${\cal
F}$-measurable, for all $t\in[0,T].$} is said to be an admissible
control for Player 1 along a partition
$\Pi=\{0=t_0<t_1<\dots<t_n=T\}$ of the interval $[0,T]$, if, for any
$j\, (1\le j\le n),$ its restriction $u_{|[t_{j-1},t_j)}$ to the
interval $[t_{j-1},t_j)$ is of the form
$u_{|[t_{j-1},t_j)}=\sum_{k\ge 1}I_{\Gamma_{j,k}} u^{j,k}$, where
$(\Gamma_{j,k})_{k\ge 1}\subset{\cal F}_{j-1}$ is a partition of
$\Omega$ and $(u^{j,k})_{k\ge 1}\subset
L_{\mathcal{G}_{j}}^0(t_{j-1}, t_{j}; U)$. If this is the case, we
write $u\in{\cal U}_{0,T}^\Pi$.

Similarly, we say that $v\in L_{\cal F}^0(0,T;V)$ is an admissible
control along the partition $\Pi$ for Player 2, if, for any $j\,
(1\le j\le n),$ its restriction $v_{|[t_{j-1},t_j)}$ to the interval
$[t_{j-1},t_j)$ is of the form $v_{|[t_{j-1},t_j)}=\sum_{k\ge
1}I_{\Gamma_{j,k}} v^{j,k}$, where $(\Gamma_{j,k})_{ k\ge
1}\subset{\cal F}_{j-1}$ is a partition of $\Omega$ and
$(v^{j,k})_{k\ge 1}\subset L_{\mathcal{H}_{j}}^0(t_{j-1}, t_{j};
V)$. If this is the case, we write $v\in{\cal V}_{0,T}^\Pi$.

Finally, for $0\le t\le t_l\in \Pi,$ we put

\centerline{${\cal U}_{t,t_l}^\Pi:=\{(u_{s})_{s\in[t, t_{l}]}|u\in
\mathcal{U}_{0,T}^{\Pi}\}$ and ${\cal
V}_{t,t_{l}}^{\Pi}:=\{(v_{s})_{s\in[t,
t_{l}]}|v\in\mathcal{V}_{0,T}^{\Pi}\}.$}
\end{definition}

Let us describe now the dynamics of our differential game along a
partition $\Pi$ of the interval $[0,T]$. For this we consider a
bounded continuous function $f=(f(t,x,u,v)):[0,T]\times R^d\times
U\times V\longrightarrow R^d$ which is supposed to be Lipschitz in
$x$, uniformly with respect to $(t,u,v)$. Given initial data
$(t,x)\in [0,T]\times R^d$ and two controls $u\in{\cal U}_{t,T}^\Pi$
and $v\in{\cal V}_{t,T}^\Pi$, we define the continuous process
$X^{t,x,u,v}=(X_{s}^{t,x,u,v})_{s\in[t,T]}$ as the unique solution
of the following pathwise differential equation: \be\label{3.1}
X_{s}^{t,x,u,v}=x+\int_{t}^s f(r,X_{r}^{t,x,u,v},u_{r},v_{r}){\rm
d}r,\ \ s\in[t,T],\ \
(u,v)\in\mathcal{U}_{t,T}^{\Pi}\times\mathcal{V}_{t,T}^{\Pi}. \ee

\noindent We remark that standard estimates show

\begin{lemma}\label{estimates-X}
For a suitable real constant $C$ independent of the partition $\Pi$
we have, for all $(u,v)\in{\cal U}_{t,T}^\Pi\times {\cal
V}_{t,T}^\Pi$, for all $(t,x),(t',x')\in [0,T]\times R^d$ and all
$s\in[t\vee t',T]$,
\be\label{1.1}\begin{array}{lll} &{\rm (i)}\ |X_{s}^{t,x,u,v}-x|\leq  CT,\\
&{\rm (ii)}\ |X_{s}^{t,x,u,v}-X_{s}^{t',x',u,v}| \leq C(|t-t'|+|x-x'|).
\end{array}\ee
\end{lemma}

Let $g:R^d\rightarrow R$ be a bounded Lipschitz function. For a game
over the time interval $[t,T]$ along the partition
$\Pi=\{0=t_0<t_1<\dots<t_n=T\}$, with $0\le i\le n-1$ such that
$t_{i}\le t<t_{i+1}$, we consider the payoff functional
$E[g(X_{T}^{t,x,u,v})|{\cal F}_{i}]$ which Player 1 tries to
maximize through the control $u\in{\cal U}_{t,T}^\Pi$ and Player 2
tries to minimize through his choice of $v\in{\cal V}_{t,T}^\Pi$.
However, in order to guarantee the existence of a value of the game,
we consider a game in which both players use non-anticipative
strategies with delay (NAD-strategies).

\begin{definition}\label{NAD}(NAD-strategies) Let
$\Pi=\{0=t_0<t_1<\dots<t_n=T\}$ be a partition of the interval
$[0,T]$ and $0\le t \le t_l\in \Pi.$\ We say that $\beta:{\cal U}_{t,t_l}^\Pi\longrightarrow {\cal
V}_{t,T}^\Pi$ is an NAD-strategy for Player 2 over the time interval
$[t,t_l]$ along the partition $\Pi$, if for all
$\mathbb{F}$-stopping time $\tau:\Omega\rightarrow
\{0,1,\dots,n-1\}$ and all controls $u,u'\in{\cal U}_{t,t_l}^\Pi$
with $u=u',$ $dsdP$-a.s. on $[[t,t_\tau]]\footnote{The stochastic
interval $[[t,t_\tau]]$ is defined as
$\{(s,\omega)\in[t,T]\times\Omega\, |\, t\le s\le
t_{\tau(\omega)}\}$.}$, it holds $\beta(u)=\beta(u'),$ $dsdP$-a.s. on
$[[t,t_{\tau+1}]].$ The set of all NAD-strategy for Player 2 over
$[t,t_l]$ along $\Pi$ is denoted by $\mathcal{B}_{t,t_{l}}^{\Pi}$.

In an obvious symmetric way we also introduce for Player 1 the set
of all NAD-strategies over the interval $[t,t_l]$ along $\Pi$, and
we denote it by $\mathcal{A}_{t, t_{l}}^{\Pi}$.
\end{definition}

\noindent The following result is crucial; it permits to associate
couples of NAD-strategies with couples of admissible controls.

\begin{lemma}\label{lemma 3.1} For all couple of NAD strategies $(\alpha,
\beta)\in\mathcal{A}_{t,t_{l}}^{\Pi}\times\mathcal{B}_{t,t_{l}}^{\Pi},$
there exists unique couple of admissible controls $(u,
v)\in\mathcal{U}_{t,t_{l}}^{\Pi}\times\mathcal{V}_{t,t_{l}}^{\Pi}\
\mbox{such that}\ \alpha(v)= u,\, \beta(u)= v,$ $dsdP$-a.s. on
$[t,t_l]\times\Omega.$
\end{lemma}
Although such a result is well-known for deterministic and
stochastic differential games (see, for instance, \cite{BCR04} and
\cite{BCQ11}), we want to sketch here the proof for the convenience
of the reader, because the context we study differs a bit from that of
\cite{BCR04} and \cite{BCQ11}.
\begin{proof}
Let $\Pi=\{0=t_0<t_1<\dots<t_n=T\}$ be a partition of the interval
$[0,T]$, $0\le t_i\le t<t_{i+1} \le t_l\in \Pi,$ and $(\alpha,\beta)
\in\mathcal{A}_{t, t_{l}}^{\Pi}\times \mathcal{B}_{t, t_{l}}^{\Pi}.$
Then, due to the definition of NAD strategies, $\alpha(v),\beta(u)$
restricted to the interval $[t,t_{i+1}]$ depend only on the
restrictions of the controls $v\in \mathcal{V}_{t, t_{l}}^{\Pi}$ and
$u\in \mathcal{U}_{t, t_{l}}^{\Pi}$ to the interval $[t,t_i]$. But
since this interval is empty or at most a singleton (and, hence, of
Lebesgue measure zero), $\alpha(v),\beta(u)$ restricted to the
interval $[t,t_{i+1}]$ don't depend on $v$ and $u$. Consequently,
given arbitrary $u^0\in \mathcal{U}_{t, t_{l}}^{\Pi},v^0\in
\mathcal{V}_{t, t_{l}}^{\Pi},$ we put
$u^1:=\alpha(v^0),v^1:=\beta(u^0)$, and we have

\centerline{$\alpha(v^1)=u^1,\ \  \beta(u^1)=v^1, \mbox{ on }
[t,t_{i+1}].$}

Supposing that we have constructed $(u^{j-1},v^{j-1})\in
\mathcal{U}_{t, t_{l}}^{\Pi}\times\mathcal{V}_{t, t_{l}}^{\Pi}$ such
that $\alpha(v^{j-1})=u^{j-1}$ and $\beta(u^{j-1})=v^{j-1},$
$dsdP$-a.s. on $[t,t_{i+j-1}]$, we put $u^{j}:=\beta(v^{j-1}),\,
v^{j}:=\alpha(u^{j-1}).$ Then, obviously, $(u^{j},v^{j})\in
\mathcal{U}_{t, t_{l}}^{\Pi}\times\mathcal{V}_{t, t_{l}}^{\Pi}$,
$(u^{j},v^{j})=(u^{j-1},v^{j-1})$, $dsdP$-a.s. on $[t,t_{i+j-1}],$
and because of the NAD property of the strategies $\alpha,\beta$ we
have $u^{j}=\beta(v^{j}),\, v^{j}=\alpha(u^{j})$, $dsdP$-a.s. on
$[t,t_{i+j}].$ By iterating the argument up to $j=l-i$ we obtain the
assertion of the lemma.
\end{proof}
\begin{remark}
Given a couple of NAD strategies $(\alpha,\beta)\in\mathcal{A}_{t,
t_{l}}^{\Pi}\times \mathcal{B}_{t, t_{l}}^{\Pi}$ the above Lemma
\ref{lemma 3.1} allows to define the dynamics
$X^{t,x,\alpha,\beta}=(X_{s}^{t,x,\alpha,\beta})_{s\in[t,t_l]}$
along the partition $\Pi$ over the interval $[t,t_l]$ ($t_l\in\Pi$)
through that of the couple of admissible controls $(u,v)\in
\mathcal{U}_{t, t_{l}}^{\Pi}\times \mathcal{V}_{t, t_{l}}^{\Pi}$
associated with by the relation $\alpha(v)= u,\, \beta(u)= v,$
$dsdP$-a.s. on $[t,t_l]\times\Omega.$
\end{remark}

\medskip

After the above preliminary discussion we can now introduce the
value functions of the game along a partition
$\Pi=\{0=t_0<\dots<t_n=T\}$ of the interval $[0,T].$ For the initial
data $(t,x)\in[0,T]\times R^d$ we define the lower value function
$V$ and the upper value function $U$\ along a partition
$\Pi=\{0=t_0<\dots<t_n=T\}$\ as follows:
\be\label{3.2} \begin{array}{lll}V^{\Pi}(t,x)&:=& \mbox{esssup}_{
\alpha\in\mathcal{A}_{t,t_{l}}^{\Pi}} \mbox{essinf}_{\beta\in
\mathcal{B}_{t,t_{l}}^\Pi}E[g(X_{T}^{t,x,\alpha,\beta})|\mathcal{F}_{i}],\\
U^{\Pi}(t,x)&:=& \mbox{essinf}_{\beta\in \mathcal{B}_{t,t_{l}}^\Pi}
\mbox{esssup}_{\alpha\in\mathcal{A}_{t,t_{l}}^{\Pi}}
E[g(X_{T}^{t,x,\alpha,\beta})|\mathcal{F}_{i}],\\
& & \hskip 3cm \ \mbox{ for }\ t_{i} \leq t<t_{i+1}<T\, \ (0\le i\le
n-1).
\end{array}\ee

\medskip

We emphasize that, since the lower and the upper value functions are
defined as a combination of essential supremum and essential infimum
over an indexed family of uniformly bounded,
$\mathcal{F}_{i}$-measurable random variables, also they themselves
are a priori bounded, $\mathcal{F}_{i}$-measurable random variables
(Recall the definition of the essential supremum and infimum, e.g.,
in Dunford and Schwartz \cite{DS}, Dellacherie \cite{D} or in the
appendix of Karatzas and Shreve \cite{KS2}, where a detailed
discussion is made.). However, in the next section we will show that
the lower and the upper value functions are deterministic (The
interested reader is also referred to \cite{Buckdahn-Li-2008}, where
a comparable approach, but in a completely different framework is
done for stochastic differential games with Isaacs condition.)

We also remark that we have the following statement as an immediate
consequence of Lemma \ref{estimates-X} and the fact the the function
$g$ is bounded and Lipschitz:

\begin{lemma}\label{lemma 3.2} Under our standard assumptions on the
coefficients $f$ and $g$ we have that there is some constant $L$
such that, for all $(t,x),(t',x')\in[0,T]\times R^d$ and for all
partition $\Pi$,
\be\label{3.2bis} \begin{array}{lll}& {\rm (i)} |V^{\Pi}(t,x)|\le  L,\\
& {\rm (ii)} |V^{\Pi}(t,x)-V^{\Pi}(t,x')| \leq  L|x-x'|, \, \ P\mbox{-a.s.}
\end{array}\ee
\end{lemma}

\section{Lower and upper value functions along a partition}

The objective of this section is to study the properties of the
above introduced lower and upper value functions along a partition
$\Pi=\{0=t_0<t_1<\dots<t_n=T\}$ of the interval $[0,T]$. More
precisely, we first establish a dynamic programming principle (DPP) which
on its part will allow to prove that the both value functions are
deterministic.

\bt\label{th3.2} (Dynamic Programming Principle)  Let
$\Pi=\{0=t_0<\dots<t_n=T\}$ be an arbitrary partition of the
interval $[0,T]$ and $(t,x)\in[0,T]\times R^d.$ Then, for $i,l$ such
that  $t_i\le t<t_{i+1}\le t_l,$

\be\label{3.10}\begin{array}{lll} V^{\Pi}(t,x) &=&\esssup_{\alpha\in
\mathcal{A}^\Pi_{t,t_l}} \essinf_{\beta\in
\mathcal{B}^\Pi_{t,t_l}}E[V^{\Pi}(t_l,X_{t_l}^{t,x,\alpha,\beta})\mid
\mathcal{F}_i],\\
U^{\Pi}(t,x) &=& \essinf_{\beta\in
\mathcal{B}^\Pi_{t,t_l}}\esssup_{\alpha\in \mathcal{A}^\Pi_{t,t_l}}
E[U^{\Pi}(t_l,X_{t_l}^{t,x,\alpha,\beta})\mid \mathcal{F}_i],\, \
P\mbox{-a.s.}
\end{array}\ee
\et

For the proof which will be split in two lemmas, we will restrict to
the lower value function along a partition; the proof for the upper value function along a partition uses a symmetric argument. Keeping the notations introduced in the above theorem we put
\be\label{3.11}\widetilde{V}^{\Pi}_l(t,x)=\esssup_{\alpha\in\mathcal{A}^\Pi_{t,t_l}}
\essinf_{\beta\in\mathcal{B}^\Pi_{t,t_l}} E[V^{\Pi}(t_l,
X_{t_l}^{t,x,\alpha,\beta})\mid \mathcal{F}_i].\ee \noindent We
remark that $\widetilde{V}^{\Pi}_l(t,x)$ is an ${\cal
F}_i$-measurable random variable.
\begin{lemma}\label{lemma 1a} Under our standard assumptions we have
$\widetilde{V}^{\Pi}_l(t,x)\le V^\Pi(t,x),$ $P$-a.s.
\end{lemma}
\begin{proof}
\underline{Step 1}. Let us fix arbitrarily $\varepsilon>0$. Then,
there exists $\alpha^\varepsilon_1\in\mathcal{A}^\Pi_{t,t_l}$\  such
that \be\label{3.12}\widetilde{V}^{\Pi}_l(t,x)\leq
\essinf_{\beta_1\in\mathcal{B}^\Pi_{t,t_l}}
E[V^{\Pi}(t_l,X_{t_l}^{t,x,\alpha^\varepsilon_1,\beta_1})\mid
\mathcal{F}_i]+\varepsilon,\ \mbox{P-a.s.}\ee \noindent Indeed,
setting $I_1(\alpha)=\essinf_{\beta_1\in
\mathcal{B}^\Pi_{t,t_l}}E[V^{\Pi}(t_l,X_{t_l}^{t,x,\alpha,\beta_1})\mid
\mathcal{F}_i]$, we know from the properties of the essential
supremum over a family of random variables that there is a countable
sequence $(\alpha_k)_{k\geq 1}\subset \mathcal{A}^\Pi_{t,t_l}$\ such
that \be\label{3.13}\widetilde{V}^{\Pi}_l(t,x)=\esssup_{\alpha_1\in
\mathcal{A}^\Pi_{t,t_l}} I_1(\alpha_1)=\sup_{k\geq  1}I_1(\alpha_k),\, P\mbox{-a.s.}\ee \noindent Then, obviously,
$\triangle_k:=\{\widetilde{V}^{\Pi}_l(t,x)\leq
I_1(\alpha_k)+\varepsilon\}\in \mathcal{F}_i,\ \ k\geq 1,$ and
putting  $\Gamma_k:={\Delta_k}\setminus(\bigcup_{i<k}\Delta_i),\
k\geq 1,$ we define an $(\Omega,{\cal F}_i)$-partition, i.e., a
partition of $\Omega,$ composed of elements of the $\sigma$-field
$\mathcal{F}_i$. Let us now introduce the mapping
$\alpha_1^\varepsilon:=\Sigma_{k\geq
1}I_{\Gamma_k}\alpha_k(\cdot):\, {\cal V}^\Pi_{t,t_l}\rightarrow
{\cal U}^\Pi_{t,t_l}.$ It can be easily checked that such
defined mapping belongs to $\mathcal{A}^\Pi_{t,t_l},$ and standard
arguments (see, e.g., \cite{Buckdahn-Li-2008}) allow to show that
$$\displaystyle E[V^{\Pi}(t_l, X_{t_l}^{t, x,
\alpha_1^\varepsilon, \beta_1})\mid \mathcal{F}_i]=\sum_{j\geq
1}I_{\Gamma_j}E[V^{\Pi}(t_l, X_{t_l}^{t, x, \alpha_j, \beta_1})\mid
\mathcal{F}_i],\mbox { for all }\beta_1\in\mathcal{B}^\Pi_{t,t_l}.$$
\noindent Therefore, again for all $\beta_1\in\mathcal{B}^\Pi_{t,t_l}$,
$$\begin{array}{lll}\displaystyle\label{3.14}&\widetilde{V}^{\Pi}_l(t,x)\leq \sum_{k\geq
1}I_{\Gamma_k}I_1(\alpha_k)+\varepsilon\\
&\displaystyle\le \sum_{k\geq 1}I_{\Gamma_k}E[V^{\Pi}(t_l,
X_{t_l}^{t, x, \alpha_k, \beta_1})\mid
\mathcal{F}_i]+\varepsilon=E[V^{\Pi} (t_l, X_{t_l}^{t, x,
\alpha_1^\varepsilon, \beta_1})\mid \mathcal{F}_i]+\varepsilon.\\
\end{array}$$

\noindent Given now an arbitrary $\beta\in\mathcal{B}^\Pi_{t,T}$ and any $u_2
\in{\cal U}^\Pi_{t_l,T}$ we make the following particular choice
of $\beta_1$:
$$\beta_1(u_1)(s):=\beta(u)(s),\, s\in[t,t_l],\ u_1\in{\cal U}^\Pi_{t,t_l},$$

\noindent where

\centerline{$\displaystyle u(s):=\left\{\begin{array}{cc}
                         u_1(s),& s\in[t,t_l] \\
                         u_2(s),& s\in (t_l,T].
                      \end{array}\right.$}
\smallskip

\noindent Abbreviating, in what follows we will write for such a
process composed over different intervals:

\centerline{$u=u_1\oplus u_2$,\, $\beta_1(u_1)
=\beta(u_1\oplus u_2)_{|[t,t_l]}$.}

\noindent We observe that $\beta_1\in{\cal B}^\Pi_{t, t_l}$, and as
consequence of its nonanticipativity property,  it is independent of
the particular choice of $u_2.$ Consequently,

\be\label{3.15}\widetilde{V}^{\Pi}_l(t,x)\leq \varepsilon +
E[V^{\Pi}(t_l, X_{t_l}^{t,x,\alpha_1^\varepsilon,\beta_1})\mid
\mathcal{F}_i],\, P\mbox{-a.s.},\ee
\noindent for our particular choice of $\beta_1$, since we have seen that this
relation holds true for all $\beta_1\in{\cal B}^\Pi_{t,t_l}$.

\medskip

\underline{Step 2}. Let us now continue by discussing the expression
$V^{\Pi}(t_l, X_{t_l}^{t,x, \alpha_1^\varepsilon,\beta_1})$ inside
the above conditional expectation in (\ref{3.15}). For this end we consider a
partition $(O_j)_{j\geq 1}$ of $R^d$, composed of nonempty Borel
sets, such that, for all $j\ge 1,$ the maximal distance between two
elements of ${\cal O}_j$ is less than or equal to $\varepsilon$. Let
us fix in all ${\cal O}_j$ an arbitrary element $y_j.$

In analogy to Step 1 we see also here that, for every $j\ge 1,$
there exists $\alpha^{\varepsilon, j}_2\in \mathcal{A}^\Pi_{t_l,T}$\
such that
$$\begin{array}{lll}
&\displaystyle V^{\Pi}(t_l,y_j)=\esssup_{\alpha_2\in
\mathcal{A}^\Pi_{t_l, T}} \essinf_{\beta_2\in \mathcal{B}^\Pi_{t_l,
T}}E[g(X_T^{t_l,y_j,\alpha_2,\beta_2})\mid \mathcal{F}_l]\\
&\displaystyle\leq \varepsilon+\essinf_{\beta_2\in \mathcal{B}^\Pi_{t_l, T}}
E[g(X_T^{t_l,y_j,\alpha^{\varepsilon, j}_2,\beta_2})\mid
\mathcal{F}_l],\  P\mbox{-a.s.}\\
\end{array}$$

\noindent In dependence of our $\beta\in{\cal B}^\Pi_{t,T}$ already
chosen in the preceding Step 1 we want to make now a particular
choice of $\beta_2\in \mathcal{B}^\Pi_{t_l, T}$. For this end we notice
that, since $(\alpha_1^\varepsilon, \beta_1)\in
\mathcal{A}^\Pi_{t,t_l} \times \mathcal{B}^\Pi_{t,t_l}$, due to
Lemma \ref{lemma 3.1} there exists a unique couple
$(u_1^\varepsilon,v_1^\varepsilon)\in\mathcal{U}^\Pi_{t,t_l}\times
\mathcal{V}^\Pi_{t,t_l}$ such that
$\alpha_1^\varepsilon(v_1^\varepsilon)=u_1^\varepsilon,$ and
$\beta_1(u_1^\varepsilon)=v_1^\varepsilon.$ With the notations
introduced in Step 1 we define now

\centerline{$\displaystyle \beta_2(u_2):=\beta(u_1^\varepsilon\oplus
u_2)_{|[t_l,T]},\, u_2\in{\cal U}^\Pi_{t_l,T}.$}

\noindent It is straight-forward to check that $\beta_1\in{\cal
B}^\Pi_{t_l,T}$, and, consequently, \be\displaystyle
V^{\Pi}(t_l,y_j)\leq
\varepsilon+E[g(X_T^{t_l,y_j,\alpha^{\varepsilon, j}_2,\beta_2})\mid
\mathcal{F}_l],\, P\mbox{-a.s.} \ee \noindent  Thus, from the
Lipschitz continuity of $V^{\Pi}(t_l,.)$ (see Lemma \ref{lemma 3.2})
we obtain \be\label{3.17}\displaystyle
\begin{array}{lll}
& &V^{\Pi}(t_l, X_{t_l}^{t,x,\alpha^\varepsilon_1,\beta_1})\leq
C\varepsilon + \displaystyle\sum_{j\geq1}V^{\Pi}(t_l,
y_j)I_{\{X_{t_l}^{t,x,\alpha^\varepsilon_1, \beta_1}\in
O_j\}}\\
& &\leq (C+1)\varepsilon\displaystyle
+\sum_{j\geq1}I_{\{X_{t_l}^{t,x,\alpha^\varepsilon_1,\beta_1}\in
O_j\}}E[g(X_T^{t_l,y_j,\alpha^{\varepsilon, j}_2,\beta_2})\mid
\mathcal{F}_l].\end{array}\ee Let us introduce now
$\displaystyle\alpha^\varepsilon_2:=\sum_{j\geq
1}I_{\{X_{t_l}^{t,x,\alpha_1^\varepsilon,\beta_1}\in O_j\}} \alpha^{
\varepsilon, j}_2$. It is easy to verify that $\alpha^\varepsilon_2$
belongs to $\mathcal{A}^\Pi_{t_l,T}$. On the other hand, for every
$(\alpha^{\varepsilon, j}_2, \beta_2) \in \mathcal{A}^\Pi_{t_l,T}
\times \mathcal{B}^\Pi_{t_l,T}$, there exists a unique couple
$(u^{\varepsilon, j}_2, v^{\varepsilon, j}_2)\in
\mathcal{U}^\Pi_{t_l,T}\times \mathcal{V}^\Pi_{t_l,T}$, such that

\centerline{$\alpha^{\varepsilon, j}_2(v^{\varepsilon, j}_2)=u^{
\varepsilon, j}_2,\ \
 \beta_2(u^{\varepsilon, j}_2)=v^{\varepsilon, j}_2,$}

\noindent and with its help we define

\centerline{$\displaystyle (u_2^\varepsilon,v_2^\varepsilon):=
\sum_{j\geq1}I_{\{X_{t_l}^{t,x,\alpha^\varepsilon_1,\beta_1}\in
O_j\}}(u^{\varepsilon, j}_2,v^{\varepsilon, j}_2)\in
\mathcal{U}^\Pi_{t_l,T} \times \mathcal{V}^\Pi_{t_l,T}$.}

\noindent Then, according to the definition of
$\alpha^\varepsilon_2$ and the nonanticipativity of the elements of
${\cal A}^\Pi_{t_l,T}$ (see Definition \ref{NAD} for nonanticipative
strategies), since $v_2^\varepsilon=v^{\varepsilon, j}_2$\ on
$\{X_{t_l}^{t,x,\alpha^\varepsilon_1,\beta_1}\in
O_j\}\times[t_l,T]$, we also have
$$\alpha_2^\varepsilon(v^{\varepsilon}_2)=\alpha^{
\varepsilon, j}_2(v^{\varepsilon}_2)=\alpha^{\varepsilon,
j}_2(v^{\varepsilon, j}_2) =u^{\varepsilon, j}_2=u_2^\varepsilon\ \mbox{on}\ \{X_{t_l}^{t,x,\alpha^\varepsilon_1,\beta_1}\in
O_j\}\times[t_l,T],\, j\ge 1.$$

\noindent Consequently, since $({\cal O}_j)_{j\ge 1}$ forms a
partition of $R^d$, it holds
$\alpha_2^\varepsilon(v^{\varepsilon}_2)=u_2^\varepsilon$.
Analogously, we obtain $\beta_2(u_2^\varepsilon)=v_2^\varepsilon$.
Moreover, recalling that $(u_1^\varepsilon, v_1^\varepsilon)\in
\mathcal{A}^\Pi_{t,t_l} \times \mathcal{B}^\Pi_{t,t_l}$ has been
introduced such that $\alpha_1^\varepsilon
(v_1^\varepsilon)=u_1^\varepsilon,\ \beta_1(u_1^\varepsilon)
=v_1^\varepsilon,$ we define a couple of controls $(u^\varepsilon,
v^\varepsilon)\in \mathcal{U}^\Pi_{t,T}\times \mathcal{V}^\Pi_{t,T}$
by putting $u^\varepsilon:=u_1^\varepsilon\oplus u_2^\varepsilon$
and $v^\varepsilon:=v_1^\varepsilon\oplus v_2^\varepsilon$.
Furthermore, we introduce
$$\alpha^\varepsilon(v):=\alpha_1^\varepsilon(v_1)\oplus
\alpha_2^\varepsilon(v_2),\ \mbox{for}\ v_1:=v_{|[t,t_l]},\ v_2:=v_{|[t_l,T]},\ v\in
\mathcal{V}^\Pi_{t,T}.$$

\noindent Then  $\alpha^\varepsilon\in \mathcal{A}^\Pi_{t,T}$, and
$\alpha^\varepsilon(v^\varepsilon)=\alpha_1^\varepsilon(v_1^\varepsilon)\oplus
\alpha_2^\varepsilon(v_2^\varepsilon)=u_1^\varepsilon \oplus
u_2^\varepsilon=u^\varepsilon$, and, on the other hand, recalling the definition of
$\beta_1$ and $\beta_2,$ we have
$$\beta(u^\varepsilon)=\beta(u_1^\varepsilon \oplus
u_2^\varepsilon)\mid _{[t,t_l)}\oplus \beta(u_1^\varepsilon \oplus
u_2^\varepsilon)\mid _{[t_l,T]}=\beta_1(u_1^\varepsilon)\oplus
\beta_2(u_2^\varepsilon)=v_1^\varepsilon \oplus
v_2^\varepsilon=v^\varepsilon.$$

\noindent This shows that
$(u^\varepsilon,v^\varepsilon)\in\mathcal{U}^\Pi_{t,T}\times
\mathcal{V}^\Pi_{t,T}$ is the unique couple of controls which is
associated with $(\alpha^\varepsilon,\beta)\in \mathcal{A}^\Pi_{t,T}
\times\mathcal{B}^\Pi_{t,T}.$ Hence, \be\label{3.19}
X_T^{t_l,X_{t_l}^{t,x,\alpha^\varepsilon_1,\beta_1},u_2^\varepsilon,
v_2^\varepsilon}=X_T^{t_l,X_{t_l}^{t,x,u^\varepsilon_1,v_1^\varepsilon},
u_2^\varepsilon,v_2^\varepsilon}=X_T^{t,x,u^\varepsilon,v^\varepsilon}
=X_T^{t,x,\alpha^\varepsilon,\beta},\ee \noindent and, taking into account in
addition the Lipschitz property of $g$ , we get
\be\label{3.18}\begin{array}{lll} & &\sum_{j\geq
1}I_{\{X_{t_l}^{t,x,\alpha^\varepsilon_1,\beta_1}\in
O_j\}}g(X_T^{t_l,y_j,\alpha^{\varepsilon, j}_2,\beta_2})\\
& &=\sum_{j\geq 1}I_{\{X_{t_l}^{t,x,\alpha^\varepsilon_1,\beta_1}\in
O_j\}}g(X_T^{t_l,y_j,u^{\varepsilon, j}_2,v^{\varepsilon, j}_2})\\
& &=\sum_{j\geq 1}I_{\{X_{t_l}^{t,x,\alpha^\varepsilon_1,\beta_1}\in
O_j\}}g(X_T^{t_l,y_j,u^\varepsilon_2,v_2^\varepsilon})\\
& &\leq g(X_T^{t_l,X_{t_l}^{t,x,\alpha^\varepsilon_1,\beta_1},
u_2^\varepsilon,v_2^\varepsilon})+C\varepsilon=
g(X_T^{t,x,\alpha^\varepsilon,\beta})+C\varepsilon.
\end{array}\ee
Consequently, from (\ref{3.17}) and (\ref{3.18}), \be \label{3.20}
V^{\Pi}(t_l, X_{t_l}^{t,x,\alpha^\varepsilon_1,\beta_1})\leq
C\varepsilon + E[g(X_T^{t,x,\alpha^\varepsilon,\beta})\mid
\mathcal{F}_l],\ P\mbox{-a.s.} \ee Furthermore, from (\ref{3.15}),
\be \label{3.21}\widetilde{V}^{\Pi}_l(t,x)\leq \varepsilon +
E[V^{\Pi}(t_l, X_{t_l}^{t,x,\alpha_1^\varepsilon,\beta_1})\mid
\mathcal{F}_i]\leq C\varepsilon +
E[g(X_T^{t,x,\alpha^\varepsilon,\beta})\mid \mathcal{F}_i],\
P\mbox{-a.s.}\ee \noindent This relation holds true for our
arbitrarily chosen and, hence, for all $\beta \in
\mathcal{B}^\Pi_{t,T}.$ It follows that
\be\label{3.22}\begin{array}{lll} \widetilde{V}^{\Pi}_l(t,x) &\leq
&C\varepsilon+\essinf_{\beta\in \mathcal{B}^\Pi_{t,t_l}}
E[g(X_T^{t,x,\alpha^\varepsilon,\beta})\mid
\mathcal{F}_i]\\
&\leq& C\varepsilon+\esssup_{\alpha\in\mathcal{A}^\Pi_{t,t_l}}
\essinf_{\beta\in \mathcal{B}^\Pi_{t,t_l}}E[g(X_T^{t,x,\alpha,
\beta})\mid \mathcal{F}_i]\\
&=& C\varepsilon+V^\Pi(t,x),\ \ P\mbox{-a.s.},
\end{array}\ee \noindent and the
statement follows by letting $\varepsilon$ tend to zero.
\end{proof}

\bigskip

Let us now come the converse statement to Lemma 3.1.
\begin{lemma} Under our standard assumptions we have $\widetilde{V}^{\Pi}_l(t,x)\ge
V^\Pi(t,x),$ $P$-a.s.
\end{lemma}
\begin{proof} Because of the symmetry of some arguments to those
in the proof of Lemma 3.1, this proof here will be kept
shorter.

Let us fix any $\alpha \in \mathcal{A}^\Pi_{t,T}$ and, for some
arbitrarily chosen $v_2\in \mathcal{V}_{t_l,T},$ we put
$\alpha_1(v_1):=\alpha(v_1 \oplus v_2)\mid _{|[t,t_l)}$,  $v_1\in
\mathcal{V}_{t,t_l}$. Obviously, such defined mapping $\alpha_1$
belongs to $\mathcal{A}^\Pi_{t,t_l}$ and, as a consequence of its
nonanticipativity, it doesn't depend on the choice of $v_2$. Thus,
from the definition of $\widetilde{V}^\Pi_l(t,x)$ it follows that

\be\label{3.23}\widetilde{V}^\Pi_l(t,x)\geq \essinf_{\beta_1\in
\mathcal{B}^\Pi_{t,t_l}}E[V^\Pi(t_l,X_{t_l}^{t,x,\alpha_1,\beta_1})\mid
\mathcal{F}_i],\, P\mbox{-a.s.},\ee

\noindent  and, similarly to (\ref{3.15}), we can show that, for any
given $\varepsilon>0$, there exists some $\beta^\varepsilon_1\in
\mathcal{B}^\Pi_{t,t_l}$\  such that

\be\label{3.24} \widetilde{V}_{l}^{\Pi}(t,x)\geq
E[V^\Pi(t_{l},X_{t_{l}}^{t,x,\alpha_{1},\beta_{1}^{\varepsilon}})
|\mathcal{F}_{i}]-\varepsilon,\, P\mbox{-a.s.}\ee

In analogy to the proof of Lemma 3.1 we discuss now the
expression
$V^\Pi(t_{l},X_{t_{l}}^{t,x,\alpha_{1},\beta_{1}^{\varepsilon}})$
inside the above conditional expectation in (\ref{3.24}). For this we let
$(u_1^\varepsilon,v_1^\varepsilon)\in\mathcal{U}^\Pi_{t,t_l}\times
\mathcal{V}^\Pi_{t,t_l}$ denote the unique couple of admissible
controls associated with $(\alpha_1, \beta_1^\varepsilon)\in
\mathcal{A}^\Pi_{t,t_l} \times \mathcal{B}^\Pi_{t,t_l}$ by Lemma
\ref{lemma 3.1}, i.e., such that
$\alpha_1(v_1^\varepsilon)=u_1^\varepsilon,$\
$\beta_1^\varepsilon(u_1^\varepsilon)=v_1^\varepsilon,$ and we
introduce the NAD-strategy $\alpha_2^\varepsilon\in
\mathcal{A}^\Pi_{t_l, T}$ by putting
$\alpha_2^\varepsilon(v_2):=\alpha(v_1^\varepsilon\oplus v_2)\mid
_{[t_l,T]},\, v_2\in{\cal V}^\Pi_{t_l, T}.$ In order to construct an
appropriate NAD-strategy $\beta_2^\varepsilon\in
\mathcal{B}^\Pi_{t_l, T},$ we use the Borel partition $(\mathcal
{O}_{j})_{j\geq 1}$ and the sequence $y_{j}\in{\cal O}_j, \, j\ge
1,$ introduced in the second step of the proof of Lemma 3.1. Choosing $\beta_2^{\varepsilon,j}\in{\cal B}^\Pi_{t_l,T}$
such that

\be\label{3.25-1}\begin{array}{lll}
V^{\Pi}(t_{l},y_{j})&\geq &\essinf_{\beta_{2}\in \mathcal {B}_{t_{l},
T}^{\Pi}}E[g(X_{T}^{t_{l},y_{j},\alpha_{2}^{\varepsilon},\beta_{2}})
|\mathcal {F}_{{l}}]\\
&\geq & E[g(X_{T}^{t_{l},y_{j},\alpha_{2}^{\varepsilon},\beta_{2}^{
\varepsilon,j}})|\mathcal {F}_{{l}}]-\varepsilon,\, P\mbox{-a.s.},\,
j\ge 1,
\end{array}\ee

\noindent we define $\beta_{2}^{\varepsilon}\in \mathcal
{B}_{t_{l},T}^{\Pi}$ and $\beta^{\varepsilon}\in \mathcal
{B}_{t,T}^{\Pi}$ by putting

\be\label{3.26}\begin{array}{lll}
&&\beta_{2}^{\varepsilon}(u_2):=\displaystyle\sum_{j\geq
1}I_{\{X_{t_{l}}^{t,x, \alpha_{1},\beta_{1}^{\varepsilon}}\in
\mathcal {O}_{j}\}}
\beta_{2}^{\varepsilon,j}(u_2),\, u_2\in \mathcal {U}_{t_{l},T}^{\Pi},\\
&&\beta^{\varepsilon}(u):=\beta_{1}^{\varepsilon}(u_{1})\oplus
\beta_{2}^{\varepsilon}(u_{2}),\
u_{1}:=u_{|[t,t_l)},\,u_{2}:=u_{|[t_l,T]},\, u\in \mathcal
{U}_{t,T}^{\Pi}.
\end{array}\ee

\noindent Consequently, taking into account the Lipschitz property
of $V^\Pi(t_l,.)$ and using (\ref{3.25-1}), we have similarly to Step
2 of the proof of the preceding Lemma 3.1

\be\label{3.27bis}\begin{array}{lll} \widetilde{V}_{l}^{\Pi}(t,x)
&\geq
&E[V^\Pi(t_{l},X_{t_{l}}^{t,x,\alpha_{1},\beta_{1}^{\varepsilon}})
|\mathcal{F}_{i}]-\varepsilon\\
&\geq &\displaystyle\sum_{j\geq
1}E[I_{\{X_{t_{l}}^{t,x,\alpha_{1},\beta_{1}^{\varepsilon}}
\in \mathcal {O}_{j}\}}V^{\Pi}(t_{l},y_{j})|\mathcal {F}_{i}]-C\varepsilon\\
&\geq &\displaystyle\sum_{j\geq
1}E[I_{\{X_{t_{l}}^{t,x,\alpha_{1},\beta_{1}^{\varepsilon}} \in
\mathcal {O}_{j}\}}g(X_{T}^{t_{l},y_{j},\alpha_{2}^{\varepsilon},
\beta_{2}^{\varepsilon,j}})|\mathcal {F}_{i}]-C\varepsilon\\
&=&\displaystyle\sum_{j\geq
1}E[I_{\{X_{t_{l}}^{t,x,\alpha_{1},\beta_{1}^{\varepsilon}} \in
\mathcal {O}_{j}\}}g(X_{T}^{t_{l},y_{j},\alpha_{2}^{\varepsilon},
\beta_{2}^{\varepsilon}})|\mathcal {F}_{i}]-C\varepsilon\\
&\geq
&E[g(X_{T}^{t_{l},X_{t_{l}}^{t,x,\alpha_{1},\beta_{1}^{\varepsilon}},
\alpha_{2}^{\varepsilon},\beta_{2}^{\varepsilon}})|\mathcal
{F}_{i}]-C\varepsilon,\, P\mbox{-a.s.}
\end{array}\ee

\noindent Let $(u_2^\varepsilon,v_2^\varepsilon)\in{\cal
U}^\Pi_{t_l,T}\times{\cal V}^\Pi_{t_l,T}$ be the unique couple of
controls associated with $(\alpha_2^\varepsilon,
\beta_2^\varepsilon)\in{\cal A}^\Pi_{t_l,T}\times{\cal
B}^\Pi_{t_l,T}$ by Lemma \ref{lemma 3.1}. Then, it is
straight-forward to show that the couple
$(u^\varepsilon,v^\varepsilon)=(u_1^\varepsilon\oplus
u_2^\varepsilon,v_1^\varepsilon\oplus v_2^\varepsilon)\in{\cal
U}^\Pi_{t,T}\times{\cal V}^\Pi_{t,T}$ verifies
$\alpha(v^\varepsilon)=u^\varepsilon,\, \beta^(u^\varepsilon)=
v^\varepsilon.$ Consequently,

\be\label{3.27}\begin{array}{lll} \widetilde{V}_{l}^{\Pi}(t,x) &\geq
&E[g(X_{T}^{t_{l},X_{t_{l}}^{t,x,\alpha_{1},\beta_{1}^{\varepsilon}},
\alpha_{2}^{\varepsilon},\beta_{2}^{\varepsilon}})|\mathcal
{F}_{i}]-C\varepsilon\\
&\geq &E[g(X_{T}^{t_{l},X_{t_{l}}^{t,x,u^\varepsilon_{1},v_{1}^{
\varepsilon}},u_{2}^{\varepsilon},v_{2}^{\varepsilon}})|\mathcal
{F}_{i}]-C\varepsilon\\
&=&E[g(X_{T}^{t,x,u^\varepsilon,v^{\varepsilon}})|\mathcal {F}_{i}]-C\varepsilon\\
&=&E[g(X_{T}^{t,x,\alpha,\beta^{\varepsilon}})|\mathcal {F}_{i}]-C\varepsilon\\
&\geq &\essinf_{\beta\in\mathcal
{B}_{t,T}^{\Pi}}E[g(X_{T}^{t,x,\alpha,\beta}) |\mathcal
{F}_{i}]-C\varepsilon,\ P\mbox{-a.s.}
\end{array}\ee

\noindent Taking into account the arbitrariness of  $\alpha\in
\mathcal {A}_{t,T}^{\Pi}$ and of $\varepsilon>0,$ we conclude
\be\label{3.28}\widetilde{V}_{l}^{\Pi}(t,x)\geq V^{\Pi}(t,x),\,
P\mbox{-a.s.},\ee and the proof is complete.
\end{proof}
Obviously, the proof of the DPP for
$V^\Pi$ is an immediate consequence of the both preceding lemmas,
and the proof for $U^\Pi$ is analogous.

\smallskip

After having established the DPP for
the lower and the upper value functions along a partition $V^\Pi$ and $U^\Pi$, we can
now show that these a priori random fields are deterministic. More
precisely, we have the following
\begin{theorem}\label{th3.1}
For all partition $\Pi$ of the interval $[0,T]$, the lower value function along a partition
$V^{\Pi}$\ as well as the upper one $U^{\Pi}$ is deterministic, i.e.,
for all $(t,x)\in[0,T]\times R^d,$
$$V^\Pi(t,x)=E\left[V^{\Pi}(t,x)\right]\,  \mbox{
and }\, U^\Pi(t,x)=E\left[U^{\Pi}(t,x)\right],\, P\mbox{-a.s.}$$
\end{theorem}
\begin{remark} The above theorem allows to identify the lower and the upper
value functions along a partition with their deterministic versions:
$V^\Pi(t,x):=E\left[V^{\Pi} (t,x)\right]$ and
$U^\Pi(t,x):=E\left[U^{\Pi}(t,x)\right],\, (t,x)\in[0,T]\times R^d.$
\end{remark}
In view of the symmetry of the arguments we will restrict
the proof to the case of the lower value function along a partition $V^\Pi.$
We consider a partition of the interval $[0,T]$ of the
form $\Pi=\{0=t_{0}<\cdots<t_{n-1}<t_{n}=T\}$ and prove by
backward iteration that the lower value function along a partition $V^{\Pi}$ is
deterministic. For this we note that, for the first step of the
backward iteration, we have the following
\begin{lemma}\label{Theorem 3.1bis}
For the above introduced partition $\Pi$ and with the above notations
we have that
$$
V^{\Pi}(t_{n-1},x)=
\mbox{esssup}_{\alpha\epsilon\mathcal{A}_{t_{n-1},t_{n}}^{\Pi}}
\mbox{essinf}_{\beta\epsilon \mathcal{B}_{t_{n-1},t_{n}}^\Pi}
E[g(X_{t_{n}}^{t_{n-1},x,\alpha,\beta})|\mathcal{F}_{n-1}]
$$
is deterministic, i.e.,
$V^{\Pi}(t_{n-1},x)=E\left[V^{\Pi}(t_{n-1},x)\right]$, $P$-a.s., for
all $x\in R^d.$
\end{lemma}
\begin{proof} A crucial role will be played by the following auxiliary
statement:

\smallskip

\noindent  Let $\tau:\Omega\rightarrow\Omega,\ \  \omega\rightarrow
\tau(\omega)=(\tau(\omega)_k)_{k\ge 1},$ be an arbitrary measurable
bijection which law $P\circ[\tau]^{-1}$ is equivalent to the
underlying probability measure $P$, such that
$\tau'(\omega):=(\tau(\omega)_1,\dots, \tau(\omega)_{n-1}),\,
\omega\in\Omega$, is ${\cal F}_{n-1}-{\cal
B}(([0,1]^2)^{n-1})$-measurable, and $\tau(\omega)_k=\omega_k,\,
k\ge n, \omega\in\Omega.$ Then
$$V^{\Pi}(t_{n-1},x)\circ \tau =V^{\Pi}(t_{n-1},x),\,
P\mbox{-a.s.}$$

\noindent Let us prove this assertion. For this we notice first
that, using the equivalence between $P\circ[\tau]^{-1}$ and $P$ as
well as the bijectivity of $\tau$, we can change the order between
$\mbox{esssup}_{\alpha\in\mathcal{A}_{t_{n-1},t_{n}}^{\Pi}}
\mbox{essinf}_{\beta\in \mathcal{B}_{t_{n-1},t_{n}}^\Pi}$ and the
transformation $\tau$ (The reader interested in details is referred
to the corresponding proof in \cite{Buckdahn-Li-2008}.), i.e., we
have
$$
V^{\Pi}(t_{n-1},x)\circ\tau=\mbox{esssup}_{\alpha\in\mathcal{A}_{t_{n-1},
t_{n}}^{\Pi}}\mbox{essinf}_{\beta\in
\mathcal{B}_{t_{n-1},t_{n}}^\Pi}
\left(E[g(X_{t_{n}}^{t_{n-1},x,\alpha,\beta})|\mathcal{F}_{n-1}]\circ\tau\right),
\ \ P\mbox{-a.s.}
$$
Let us study now the expression
$E[g(X_{t_{n}}^{t_{n-1},x,\alpha,\beta})|\mathcal{F}_{n-1}](\tau),$
occurring in the above formula. For this we recall first that, due
to the definition, for any couple of admissible control processes
$(u, v)\in\mathcal{U}_{t_{n-1},t_{n}}^{\Pi}
\times\mathcal{V}_{t_{n-1},t_{n}}^{\Pi},$\ there exists an $(\Omega, \mathcal{F}_{n-1})$-partition $(\Gamma_{j})_{j\geq 1}$
and an associated sequence of couples of control processes $(u_{j},
v_{j})\in L_{\mathcal{G}_{n}}^{0} (t_{n-1},t_{n}; {U})\times
L_{\mathcal{H}_{n}}^{0} (t_{n-1},t_{n}; {V}),\ \ j\geq 1,$
such that $(u, v)=\sum_{j\geq 1}I_{\Gamma_{j}}(u^{j}, v^{j}).$ Since
$\Gamma_j\in{\cal F}_{n-1},$ we can find a Borel function $f_j$ with
$f_j(\zeta_1,\dots,\zeta_{n-1})=I_{\Gamma_j},\ j\ge 1.$ Then the
relation
$$I_{\tau^{-1}(\Gamma_j)}(\omega)=f_j(\tau'(\omega)),\,
\omega\in \Omega,$$

\smallskip

\noindent  proves that $\tau^{-1}(\Gamma_j)\in{\cal F}_{n-1},\, j\ge
1.$ Hence, taking into account that the mapping
$\tau:\Omega\rightarrow\Omega$ is bijective and
$\tau(\omega)_n=\omega_n,\ \omega\in\Omega$, we see that also
$(\tau^{-1}(\Gamma_{j}))_{j\geq 1}$ forms an $(\Omega,
\mathcal{F}_{n-1})$-partition, and  \be\label{3.6} (u(\tau),
v(\tau))=\sum_{j\geq 1}I_{\Gamma_{j}}(\tau)\cdot(u^{j},
v^{j})=\sum_{j\geq 1}I_{\tau^{-1}(\Gamma_{j})}\cdot(u^{j},
v^{j})\in\mathcal{U}_{t_{n-1}, t_{n}}^{\Pi}\times
\mathcal{V}_{t_{n-1}, t_{n}}^{\Pi}. \ee (Recall that $u^{j}$ is
${\cal G}_n$-measurable and, hence, a measurable function of
$\zeta_{n,1}$, while $v^{j}$ is is ${\cal H}_n$-measurable and,
thus, a measurable function of $\zeta_{n,2}$.)

\noindent On the other hand, a straight-forward application of the
transformation $\tau:\Omega\rightarrow\Omega$ to equation
(\ref{3.1}) yields
$$X_{t_{2}}^{t_{1},x,u,v}(\tau)=X_{t_{2}}^{t_{1},x,u(\tau),v(\tau)}.$$
Indeed, the only random processes in the equation (\ref{3.1}) of the dynamics
are the control processes $u$ and $v$.

Let us now consider an arbitrary couple of nonanticipative
strategies $(\alpha,\beta)\in\mathcal{A}_{t_{n-1},t_{n}}^{\Pi}\times
\mathcal{B}_{t_{n-1},t_{n}}^{\Pi}$ with which we associate the
mappings $\alpha_\tau: \mathcal{V}_{t_{n-1},t_{n}}^{\Pi}\rightarrow
\mathcal{U}_{t_{n-1},t_{n}}^{\Pi}$ and $\beta_\tau:
\mathcal{U}_{t_{n-1},t_{n}}^{\Pi}\rightarrow \mathcal{V}_{t_{n-1},
t_{n}}^{\Pi}$ defined as follows:
$$\alpha_{\tau}(v)=\alpha(v(\tau^{-1}))(\tau),\ \beta_{\tau}(u)=
\beta(u(\tau^{-1}))(\tau),$$ for
$u\in\mathcal{U}_{t_{n-1},t_{n}}^{\Pi}, v\in\mathcal{V}_{t_{n-1},
t_{n}}^{\Pi}.$ It can be easily checked that such defined
mappings are themselves again nonanticipative strategies:
$\alpha_\tau\in\mathcal{A}_{t_{n-1}, t_{n}}^{\Pi},\, \beta\in
\mathcal{B}_{t_{n-1},t_{n}}^{\Pi}$. Moreover, from the bijectivity
of $\tau$ it can be easily deduced that
$$\{\alpha_{\tau}|\alpha\in\mathcal{A}_{t_{n-1}
,t_{n}}^{\Pi}\}=\mathcal{A}_{t_{n-1}, t_{n}}^{\Pi},\
\{\beta_{\tau}|\beta\in\mathcal{B}_{t_{n-1},t_{n}}^{\Pi}\}=\mathcal{B}_{t_{n-1},
t_{n}}^{\Pi}.$$

\medskip

Given an arbitrary couple of nonanticipative strategies
$(\alpha,\beta)\in\mathcal{A}_{t_{n-1},t_{n}}^{\Pi}\times
\mathcal{B}_{t_{n-1},t_{n}}^{\Pi}$ we consider the couple of
admissible controls $(u, v)\in\mathcal{U}_{t_{n-1},t_{n}}^{\Pi}
\times\mathcal{V}_{t_{n-1},t_{n}}^{\Pi},$ associated with by the
relations $\alpha(v)=u,\, \beta(u)=v$. Since $\tau'$ is ${\cal
F}_{n-1}$-measurable and $\tau(\omega)_n=\omega_n,\
\omega\in\Omega,$ we obtain
$$\begin{array}{lll}
&\displaystyle E[g(X_{t_{n}}^{t_{n-1}, x, \alpha,
\beta})|\mathcal{F}_{n-1}]\circ\tau=E[g(X_{t_{n}}^{t_{n-1}, x, u,
v})|\mathcal{F}_{n-1}]\circ \tau=E[g(X_{t_{n}}^{t_{n-1}, x, u,
v}\circ
 \tau)|\mathcal{F}_{n-1}]\\
&\displaystyle =E[g(X_{t_{n}}^{t_{n-1}, x, u(\tau),
v(\tau)})|\mathcal{F}_{n-1}].\end{array}$$

\noindent On the other hand, we observe that, due to the definition
of the strategies $\alpha_\tau$ and $\beta_\tau$ we have
$$u=\alpha(v)=\alpha(v(\tau)\circ\tau^{-1}),\ \mbox{i.e.,}\ u(\tau)=\alpha(v(\tau)\circ\tau^{-1})(\tau)=\alpha_\tau(v(\tau)),$$

\noindent and the symmetric argument yields
$v(\tau)=\beta_\tau(u(\tau))$. Consequently, the unique couple of
admissible controls associated with $(\alpha_\tau,\beta_\tau)$ is
$(u(\tau),v(\tau))$, and we can conclude that
$$E[g(X_{t_{n}}^{t_{n-1}, x, \alpha,
\beta})|\mathcal{F}_{n-1}](\tau)=E[g(X_{t_{n}}^{t_{n-1}, x,
\alpha_{\tau}, \beta_{\tau}})|\mathcal{F}_{n-1}].$$

\noindent Using this together with the fact that
$$\{\alpha_{\tau}|\alpha\in\mathcal{A}_{t_{n-1}
,t_{n}}^{\Pi}\}=\mathcal{A}_{t_{n-1}, t_{n}}^{\Pi},\
\{\beta_{\tau}|\beta\in\mathcal{B}_{t_{n-1},t_{n}}^{\Pi}\}=
\mathcal{B}_{t_{n-1},t_{n}}^{\Pi},$$

\noindent we obtain
$$\begin{array}{lll}
&\displaystyle V^{\Pi}(t_{n-1},x)\circ\tau=\mbox{esssup}_{\alpha\in\mathcal{A}_{t_{n-1},t_{n}}^{\Pi}} \mbox{essinf}_{\beta\in
\mathcal{B}_{t_{n-1},t_{n}}^\Pi} \left(E[g(X_{t_{n}}^{t_{n-1},x,
\alpha,\beta})|\mathcal{F}_{n-1}]\circ\tau\right)\\
&\displaystyle =\mbox{esssup}_{\alpha\in\mathcal{A}_{t_{n-1},t_{n}}^{\Pi}} \mbox{essinf}_{\beta\in
\mathcal{B}_{t_{n-1},t_{n}}^\Pi} E[g(X_{t_{n}}^{t_{n-1}, x,
\alpha_{\tau}, \beta_{\tau}})|\mathcal{F}_{n-1}]\\
&\displaystyle =\mbox{esssup}_{\alpha\in\mathcal{A}_{t_{n-1},t_{n}}^{\Pi}} \mbox{essinf}_{\beta\in
\mathcal{B}_{t_{n-1},t_{n}}^\Pi} E[g(X_{t_{n}}^{t_{n-1}, x, \alpha,
\beta})|\mathcal{F}_{n-1}]\\
&\displaystyle =V^{\Pi}(t_{n-1},x).\\
\end{array}
 $$
\noindent Hence, $V^{\pi}(t_1, x)\circ \tau=V^{\pi}(t_1, x)$,
P-a.s., and the proof of Lemma 3.3 will be completed by the
following result.
\end{proof}
\bl \label{lemma 3.3} Let $\xi\in L^1(\Omega,{\cal F}_{n-1},P)$ be a
random variable which is invariant with respect to all measurable
bijection $\tau:\Omega\rightarrow\Omega$ which law
$P\circ[\tau]^{-1}$ is equivalent to the underlying probability
measure $P$, such that $\tau'(\omega):=(\tau(\omega)_1,\dots,
\tau(\omega)_{n-1}),\, \omega\in\Omega$, is ${\cal F}_{n-1}-{\cal
B}(([0,1]^2)^{n-1})$-measurable, and $\tau(\omega)_k=\omega_k,\,
k\ge n, \omega\in\Omega.$ Then $\xi$ is almost surely constant,
i.e., $\xi=E[\xi].$
\end{lemma}
\begin{proof}
We begin with noting that it suffices to prove this lemma under the
additional assumption that $\xi$ is nonnegative. Otherwise, we can always
decompose $\xi$ as a difference of its positive and its negative part, and
observing that both parts are invariant with respect to $\tau$ on their turn
we can make the proof for them separately.

Given $1\le i\le n-1,\, j=1,2,$ let us denote by $\theta_{i,j}$ the
vector of all coordinate mappings
$(\zeta_{1,1},\zeta_{1,2},\zeta_{2,1},\allowbreak \zeta_{2,2},
\dots)$ but without the component $\zeta_{i,j}$. Then, putting
$\zeta(\omega):=\omega,\, \omega\in\Omega,$ we can identify
$\zeta\equiv(\theta_{i,j},\zeta_{i,j}),$ and with this
identification we can write
$\xi(\omega)=\xi(\theta_{i,j}(\omega),\zeta_{i,j}(\omega)),\,
\omega\in\Omega.$

\smallskip

Recalling that $\xi\ge 0,$ let us now introduce the following ${\cal F}_{n-1}
$-measurable mapping
$\varphi:\Omega\rightarrow[0,1]:$
$$\varphi(\omega)=\varphi(\theta_{i,j}(\omega),\zeta_{i,j}(\omega))=
\frac{\int_0^{\zeta_{i,j}(\omega)}\left(\xi(\theta_{i,j}(\omega),s)+1\right)ds}
{\int_0^1\xi(\theta_{i,j}(\omega),s)ds+1},\,\, \omega\in\Omega.$$
Obviously, $\varphi(\theta_{i,j}(\omega),.):[0,1]\rightarrow[0,1]$
is a continuous, strictly increasing bijection which derivative is
$$\frac{\partial}{\partial s}\varphi(\theta_{i,j}(\omega),s)=
\frac{\xi(\theta_{i,j}(\omega),s)+1}
{\int_0^1\xi(\theta_{i,j}(\omega),r)dr+1},\ s\in[0,1],\ \omega\in\Omega.$$
We now put
$$\tau(\omega):=(\theta_{i,j}(\omega),\varphi(\theta_{i,j}(\omega),
\zeta_{i,j}(\omega))),\, \, \omega\in\Omega.$$
Such defined mapping $\tau:\Omega\rightarrow\Omega$
satisfies the assumptions of the lemma. Indeed, due to the
definition $\tau$ is a bijection, $\tau(\omega)_k=\omega_k,\ k\ge
n,\ \omega\in\Omega$, and $\tau'$ is ${\cal F}_{n-1}$-measurable.
Moreover, the law $P\circ[\tau]^{-1}$ is equivalent to the
underlying probability measure $P$. Indeed, for any nonnegative
random variable $\eta$ over $(\Omega,{\cal F},P)$ we have
$$\begin{array}{lll}&\displaystyle E\left[\eta(\tau)\frac{\partial}{\partial s}\varphi(\theta_{i,j},
\zeta_{i,j})\right]\displaystyle =E\left[\int_0^1\eta(\theta_{i,j},\varphi(\theta_{i,j},s))
\frac{\partial}{\partial s}\varphi(\theta_{i,j},s)ds\right]\\
&\displaystyle
=E\left[\int_0^1\eta(\theta_{i,j},s)ds\right]=E[\eta],\\
\end{array}$$
\noindent where $\displaystyle \frac{\partial}{\partial s}\varphi(\theta_{i,j},s)>0,$
for all $s\in[0,1].$ Consequently, we know from our assumption that the random
variable $\xi$ is invariant under the transformation $\tau$, and, thus, observing
that $\displaystyle\int_0^1\xi(\theta_{i,j}(\omega),s)ds$ does not depend on
$\zeta_{i,j}(\omega)$, we have
$$\begin{array}{lll}&\displaystyle E[\xi^2]+E[\xi]=E[\xi(\xi+1)]\\
&\displaystyle =E[\xi(\tau)(\xi+1)]=E\left[\xi(\tau)\frac{\partial}{\partial s}
\varphi(\theta_{i,j},\zeta_{i,j})\left(\int_0^1\xi(\theta_{i,j},s)ds+1\right)\right]\\
&\displaystyle
=E\left[\xi\left(\int_0^1\xi(\theta_{i,j},s)ds+1\right)\right]
=E\left[\xi\int_0^1\xi(\theta_{i,j},s)ds\right]+E[\xi]\\
&\displaystyle =E\left[\left(\int_0^1\xi(\theta_{i,j},s)ds\right)^2\right]+E[\xi].\\
\end{array}$$

\noindent Consequently,
$$E\left[\int_0^1\xi(\theta_{i,j},s)^2ds\right]=E[\xi^2]=E\left[\left(\int_0^1
\xi(\theta_{i,j},s)ds\right)^2\right],$$
from where we see that
$$E\left[\int_0^1\left(\xi(\theta_{i,j},s)-\int_0^1\xi(\theta_{i,j},s)ds\right)^2
ds\right]=0.$$
It follows that
$$\displaystyle \xi=\int_0^1\xi(\theta_{i,j},s)ds,\ \mbox{P-a.s.}, 1\le i\le n-1,\ j=1,2.$$

\noindent Therefore, taking into account that $\xi$ is ${\cal F}_{n-1}$-measurable
and iterating the above result, we get
$$\begin{array}{lll}&\displaystyle \xi=\int_0^1\xi(\theta_{1,1},s_{1,1})ds_{1,1} =
\int_0^1\left(\int_0^1\xi(\theta_{1,2},s_{1,2})ds_{1,2}\right)(\theta_{1,1},
s_{1,1})ds_{1,1}\\
&\displaystyle =\int_{[0,1]^2}\xi(s_1,(\zeta_{2,1},\zeta_{2,2},\dots,\zeta_{n-1,1},
\zeta_{n-1,2})ds_1=\dots =\int_{[0,1]^{2(n-1)}}\xi(s)ds,\ \mbox{P-a.s.}
\end{array}$$

\noindent The proof of the lemma is complete now.
\end{proof}
By iterating the argument developed in the both preceding lemmas, we
can prove now Theorem \ref{th3.1}.
\begin{proof} From the both preceding lemmas we see that together
with $V^\Pi(t_n,.):=g(.)$ also the function $V^\Pi(t_{n-1},.)$ is
deterministic. On the other hand, from the DPP satisfied by $V^\Pi$ we obtain
$$\displaystyle V^\Pi(t_{n-2},x)=
\mbox{esssup}_{\alpha\epsilon\mathcal{A}_{t_{n-2},t_{n-1}}^{\Pi}}
\mbox{essinf}_{\beta\epsilon \mathcal{B}_{t_{n-2},t_{n-1}}^\Pi}
E[V^\Pi(t_{n-1},X_{t_{n-1}}^{t_{n-2},x,\alpha,\beta})|\mathcal{F}_{n-2}],\
\mbox{P-a.s.},$$

\noindent for all $x\in R^d.$ Hence, applying the argument of the
both preceding lemmas again, but now for the deterministic function
$V^\Pi(t_{n-1},.)$ instead of $g$ (recall that due to Lemma
\ref{estimates-X} also the function $V^\Pi(t_{n-1},.)$ is bounded
and Lipschitz), we conclude that also the function
$V^\Pi(t_{n-2},.)$ is deterministic. Iterating this argument, we see
that all $V^\Pi(t_l,.)\, (0\le l\le n)$ are deterministic. This
implies that $V^\Pi$ is a deterministic function. Indeed, let
$t_i\le t<t_{i+1}$. For the conclusion that the non-randomness of
$V^\Pi(t_{i+1},.)$ involves that of $V^\Pi(t,.)$, it suffices to
replace the driving coefficient $f(s,x,u,v)$ of the controlled
dynamics by $f(s,x,u,v)I_{[t,T]}(s).$ This substitution doesn't
change the values of $V^\Pi(t_l,x),\, (i+1\le l\le n)$ and
$V^\Pi(t,x),\, x\in R^d,$ but now $V^\Pi(t,x)$ coincides with the
deterministic function $V^\Pi(t_i,.)$ associated with the driver
$f(s,x,u,v)I_{[t,T]}(s).$ The proof of Theorem \ref{th3.1} is
complete.
\end{proof}
The both preceding major results, the DPP as well as the statement of non-randomness yield the following
important characterization of the lower and the upper value
functions along a partiton.
\begin{theorem}\label{Remark3.3}
For all partition $\Pi$ of the time interval $[0,T]$, and all
$(t,x)\in[0,T]\times R^d,$ we have

\be\label{3.25}\begin{array}{lll} V^{\Pi}(t,x) &=&\sup_{\alpha\in
\mathcal{A}^\Pi_{t,T}} \inf_{\beta\in \mathcal{B}^\Pi_{t,T}}
E[g(X_{T}^{t,x,\alpha,\beta})],\\
U^{\Pi}(t,x) &=& \inf_{\beta\in \mathcal{B}^\Pi_{t,T}}
\sup_{\alpha\in \mathcal{A}^\Pi_{t,T}}
E[g(X_{T}^{t,x,\alpha,\beta})].
\end{array}\ee
\end{theorem}
\begin{proof}
Let $\Pi=\{0=t_0<\dots<t_n=T\}$, $t_i\le t<t_{i+1}\, (0\le i\le
n-1)$ and $x\in R^d$. Moreover, fix an arbitrary $\varepsilon>0.$
Then, due to (\ref{3.21}) from the the proof of the DPP we know that
there exists $\alpha^\varepsilon\in{\cal A}^\Pi_{t,T}$ such that,
for all $\beta \in \mathcal{B}^\Pi_{t,T}$,

\be\label{3.29}{V}^{\Pi}(t,x)(=\widetilde{V}^{\Pi}_l(t,x))\leq
E[g(X_T^{t,x,\alpha^\varepsilon,\beta})\mid
\mathcal{F}_i]+\varepsilon,\, P\mbox{-a.s.},\ee

\noindent  and from (\ref{3.27}) we get for all $\alpha\in \mathcal
{A}_{t,T}^{\Pi}$ the existence of $\beta^{\alpha,\varepsilon}\in
\mathcal{B}^\Pi_{t,T}$ such that

\be\label{3.30} {V}^{\Pi}(t,x)(=\widetilde{V}^{\Pi}_l(t,x))\geq
E[g(X_{T}^{t,x,\alpha,\beta^{\alpha,\varepsilon}})|\mathcal
{F}_{i}]-\varepsilon,\, P\mbox{-a.s.}\ee

\noindent Consequently, considering that the function $V^\Pi$ is
deterministic and taking the expectation on both sides of
(\ref{3.29}) and (\ref{3.30}), we get
$$E[g(X_{T}^{t,x,\alpha,\beta^{\alpha,\varepsilon}})]-\varepsilon
\le {V}^{\Pi}(t,x)\le E[g(X_T^{t,x,\alpha^\varepsilon,\beta})\mid
\mathcal{F}_i]+\varepsilon,$$

\noindent for all $(\alpha,\beta)\in \mathcal {A}_{t,T}^{\Pi}\times
\mathcal{B}^\Pi_{t,T}$. Thus, taking into account the arbitrariness
of $\varepsilon>0$, the statement for $V^\Pi$ follows directly, and
that for $U^\Pi$ can be verified analogously. The proof is complete.
\end{proof}

We observe that the latter Theorem3.3 combined with (\ref{1.1}) provides directly the following statement:

\bl\label{lemma 3.4} There is some real constant $L$, only depending
on the bound of $f$ and the Lipschitz constants of $f(s,.,u,v)$ and
of $g$, such that, for all partition $\Pi$ of the interval $[0,T]$
and $(t,x),(t',x')\in[0,T]\times R^d$,
\be\label{3.32}|V^{\Pi}(t,x)-V^{\Pi}(t',x')|+|U^{\Pi}(t,x)-U^{\Pi}(t',x')|\leq
L(|t-t'|+|x-x'|).\ee \el

\section{Value in mixed strategies and associated Hamilton-Jacobi-Isaacs
equation}

The objective of this section is to show that the lower and the upper
value functions along a partition $V^\Pi$, $U^\Pi$ converge, as the maximal distance
$|\Pi_n|$ between two neighbouring points of $\Pi_n$ tends to zero
as $n\rightarrow +\infty,$ and that their common limit function $V$
is the viscosity solution of the Hamilton-Jacobi-Isaacs equation

\be\label{3.33} \left \{\begin{array}{rll}
\displaystyle\frac{\partial}{\partial t}W(t,x)+\sup_{\mu\in\Delta U}
\inf_{\nu\in\Delta V}\left(\widetilde{f}(t,x,\mu,\nu)\nabla W(t,x)
\right)&=&0;\\
W(T,x)&=&g(x),
\end{array}\right.
\ee

\noindent where
$$\displaystyle\widetilde{f}(x,\mu,\nu):=\int_{U}\int_{V}
f(x,u,v)\mu(du)\nu(dv),\ \mu\in\Delta U,\ \nu\in\Delta V.$$

\noindent More precisely, our main result of this section is the
following
\begin{theorem}\label{main result}
Under our standard assumptions on the coefficients $f$ and $g$, the
above Hamilton-Jacobi-Isaacs equation (\ref{3.33}) possesses in the
class of bounded continuous functions a unique viscosity solution
$V$. Moreover, for any sequence of partitions $\Pi_n,\, n\ge 1,$ of
the interval $[0,T]$ with $|\Pi_n|\rightarrow 0$ as $n\rightarrow
+\infty,$ both the sequence of the lower value functions along a partition $V^{\Pi_n}$
as well as that of the upper value functions along a partition $U^{\Pi_n},\, n\ge 1,$
converge uniformly on compacts to the function $V$.
\end{theorem}
The definition of a continuous viscosity solution is by now
standard, and the reader interested can find many literatures, e.g., refer to~\cite{CIL}.
\bde\mbox{ } A real-valued
continuous function $W\in C([0,T]\times {\mathbb{R}}^d )$ is called \\
  {\rm(i)} a viscosity subsolution of equation (\ref{3.33}) if $W(T,x) \leq \Phi (x), \mbox{for all}\ x \in
  {\mathbb{R}}^d$, and if for all functions $\varphi \in C^1([0,T]\times
  {\mathbb{R}}^d)$ and $(t,x) \in [0,T) \times {\mathbb{R}}^d$\ such that $W-\varphi $\ attains its
  local maximum at $(t, x)$,
     $$
     \frac{\partial \varphi}{\partial t} (t,x) +\sup_{\mu\in\Delta U}
\inf_{\nu\in\Delta V}\left(\widetilde{f}(t,x,\mu,\nu)\nabla \varphi(t,x)
\right)\geq 0;
     $$
{\rm(ii)} a viscosity supersolution of equation (\ref{3.33}) if $W(T,x)
\geq \Phi (x), \mbox{for all}\ x \in
  {\mathbb{R}}^d$, and if for all functions $\varphi \in C^1([0,T]\times
  {\mathbb{R}}^d)$ and $(t,x) \in [0,T) \times {\mathbb{R}}^d$ such that $W-\varphi $\ attains its
  local minimum at $(t, x)$,
     $$
     \frac{\partial \varphi}{\partial t} (t,x) +  \sup_{\mu\in\Delta U}
\inf_{\nu\in\Delta V}\left(\widetilde{f}(t,x,\mu,\nu)\nabla \varphi(t,x)
\right) \leq 0;
     $$
 {\rm(iii)} a viscosity solution of equation (\ref{3.33}) if it is both a viscosity sub- and a supersolution of equation
     (\ref{3.33}).\ede
The whole section is devoted to the proof of the above theorem. The
proof will be split in a sequel of auxiliary statements. Let us begin with observing that the equi-Lipschitz continuity of
the families of lower and upper value functions along a partition, indexed with the
help of the partitions $\Pi$ of the interval $[0,T]$, stated in
Lemma \ref{lemma 3.4}, is crucial for the application of the
Arzel\`{a}-Ascoli Theorem. Let us arbitrarily fix a sequence of
partitions $(\Pi_n)_{n\ge 1}$ of the interval $[0, T]$, such that for the mesh of the partition $\Pi_n$ it holds:
$|\Pi_n|\rightarrow 0$ as $n\rightarrow +\infty.$ Then we have

\begin{lemma}\label{Arzela-Ascoli} There exists a subsequence of
partitions, again denoted by $(\Pi_n)_{n\ge 1}$, and there are bounded
Lipschitz functions $V,U:[0,T]\times R^d\rightarrow R$
such that $(V^{\Pi_n},U^{\Pi_n})\rightarrow(V,U),$ uniformly on
compacts in $[0,T]\times R^d.$
\end{lemma}
Later we will see that the function $(V,U)$ defined by this Lemma 4.1
coincide and are independent of the choice of the sequence of
partitions.
\begin{proof} Indeed, from the Arzel\`{a}-Ascoli Theorem we know
that, for any compact subset $K$ of $[0,T]\times R^d$ and for any
subsequence of partitions of $[0,T]$, there exist a subsequence
$(\Pi_n')$ and functions $U',V':K\rightarrow R$ such that
$(V^{\Pi_n'}, U^{\Pi_n'})\rightarrow (V',U')$ uniformly on $K$, as
$n\rightarrow +\infty.$ By combining this result with a standard
diagonalisation argument we can easily prove the stated assertion.
\end{proof}
Let us fix the subsequence $(\Pi_n)_{n\ge 1}$ related with $U,V$ by
Lemma \ref{Arzela-Ascoli}. From Lemma \ref{lemma 3.4} we have
\begin{corollary} For the real constant $L$ introduced in Lemma
\ref{lemma 3.4} we have, for all $(t,x),(t',x')\in [0, T]\times R^d$,
\be\label{3.32}|V(t,x)-V(t',x')|+|U(t,x)-U(t',x')|\leq
L(|t-t'|+|x-x'|).\ee
\end{corollary}
By taking into account the uniform boundedness of the functions
$V^\Pi,\ U^\Pi,$\ parameterized by $\Pi$-partition of the interval $[0,T]$ (Indeed, they are bounded by the bound of $g$.), this shows, in particular,
that $V,\ U\in C_b([0,T]\times R^d)$ are bounded continuous functions. We are able to prove that $V$ and $U$
are viscosity solutions of equation (\ref{3.33}). For this let us
begin with

\begin{proposition}\label{th3.3} The function $V$ is a viscosity solution
of the Hamilton-Jacobi-Isaacs equation (\ref{3.33}).
\end{proposition}
In order to prove this statement, we show in a first step that
\begin{lemma}\label{lemma1-viscosity V}
The function $V$ is a viscosity subsolution of the
Hamilton-Jacobi-Isaacs equation (\ref{3.33}).
\end{lemma}
\begin{proof} Since we know that, by definition (\ref{3.2}) of $V^\Pi$,
$V^{\Pi}(T,x)=g(x),\, x\in R^d,$ for all partition $\Pi$, we also
have $V(T,x)=g(x),\, x\in R^d.$ Let $(t,x)\in [0,T)\times R^d$ and
$\varphi\in {C}^{1}([0,T]\times \mathbb{R}^{d})$\ be an arbitrary
test function such that $\varphi(t,x)-V(t,x)=0\leq
\varphi(s,y)-V(s,y),\, (s,y)\in[0,T]\times R^d.$ Since $V\in
C_b([0,T];R^d)$ is bounded, we can assume without loss of generality
that $\varphi\in {C}^{1}_b([0,T]\times \mathbb{R}^{d})$, i.e., that
$\varphi$ itself as well as its first order derivatives are bounded.
Recall that verifying that $V$ is a viscosity subsolution is
equivalent with showing that
\be\label{3.44}\displaystyle\frac{\partial}{\partial
t}\varphi(t,x)+\sup_{\mu\in \Delta U}\inf_{\nu\in \Delta
V}\widetilde{f}(t,x,\mu,\nu)\nabla\varphi(t,x)\geq0.\ee

For proving the above relation we note that for any $\rho>0$ and
$M>0$ we can find a positive integer $n_{\rho,M}$ such that, for all
$n\geq n_{\rho,M},$
$$|\varphi(t,x)-V^{\Pi_{n}}(t,x)| \leq \rho,\mbox{ and }
V^{\Pi_{n}}(s,y)\leq \varphi(s,y)+\rho,\mbox{ for all
}s\in[0,T],\ |y|\le M.$$

\smallskip

\noindent Indeed, recall that $V^{\Pi_n}\rightarrow V$ converges
uniformly on compacts, $V(t,x)=\varphi(t,x)$ and $V\le\varphi$ on
$[0,T]\times R^d.$

Let $n\ge n_{\rho,M},$\ $\Pi_{n}=\{0=t_{0}^{n}<\cdots<t_{N}^{n}=T\}$,
and let $i=i_n$\ be such that $t_{i}^{n}\leq t<t_{i+1}^{n}\leq
t_{l}^{n}.$  Then, from the DPP (Theorem
\ref{th3.2}) with respect to the partition $\Pi_n$ and since
$V^{\Pi_n}$\ is bounded by some constant $C$, uniformly with respect to $n\ge
1$, we have

\be\label{3.34}\begin{array}{lll}
\varphi(t,x)-\rho&\leq& V^{\Pi_{n}}(t,x)\\
&=&\esssup_{\alpha\in\mathcal
{A}_{t,t_{l}^{n}}^{\Pi_{n}}}\essinf_{\beta\in\mathcal
{B}_{t,t_{l}^{n}}^{\Pi_{n}}}E[V^{\Pi_{n}}(t_{l}^{n},X_{t_{l}^{n}
}^{t,x,\alpha,\beta})|\mathcal{F}_{i}]\\
&\leq &\esssup_{\alpha\in\mathcal{A}_{t,t_{l}^{n}}^{\Pi_{n}}}
\essinf_{\beta\in\mathcal{B}_{t,t_{l}^{n}}^{\Pi_{n}}}
E[\varphi(t_{l}^{n},X_{t_{l}^{n}}^{t,x,\alpha,\beta})|\mathcal
{F}_{i}]\\
& & \hskip 3cm + C P\{\, |X_{t_{l}^{n}}^{t,x,\alpha,\beta}|>M|\mathcal{F}_{i}\}+\rho,\ P\mbox{-a.s.}\\
\end{array}\ee

\noindent However,
$$\displaystyle P\{\, |X_{t_{l}^{n}}^{t,x,\alpha,\beta}|>M|\mathcal{F}_{i}\}\leq
\frac{1}{M}E[|X_{t_{l}^{n}}^{t,x,\alpha,\beta}|\, |\mathcal
{F}_{i}]\leq \frac{1}{M}(|x|+TC_f),$$

\noindent where we have used that
$|X_{t_{l}^{n}}^{t,x,\alpha,\beta}|\le |x|+TC_f$, with $C_f$
denoting the bound of $f$. Thus, by choosing $M=M_\rho$ large
enough, such that $\frac{C}{M}(|x|+TC_f)\le \rho$, and recalling the
equation for the dynamics of $X_.^{t,x,\alpha,\beta}$, we have for
$n\ge n_\rho \, (:=n_{\rho,M_\rho})$
\be\label{3.35} \begin{array}{lll}& & \displaystyle-3\rho \leq  \esssup_{\alpha\in\mathcal
{A}_{t,t_{l}^{n}}^{\Pi_{n}}} \essinf_{\beta\in\mathcal
{B}_{t,t_{l}^{n}}^{\Pi_{n}}}\hskip -0.1cm E\hskip -0.1cm\left[
\varphi (t_{l}^{n},X_{t_{l}^{n}}^{t,x,\alpha,\beta})-\varphi(t,x)|
\mathcal{F}_{i}\right]\\
&=& \displaystyle \mbox{ess}\sup_{\alpha\in\mathcal
{A}_{t,t_{l}^{n}}^{ \Pi_{n}}} \mbox{ess}\inf_{\beta\in\mathcal
{B}_{t,t_{l}^{n}}^{\Pi_{n}}}\hskip -0.15cm E\hskip -0.1cm\left[
\hskip -0.1cm\int_{t}^{t_{l}^{n}}\hskip -2mm\left(\hskip
-0.1cm\frac{\partial}{\partial r} \varphi(r,
X_{r}^{t,x,\alpha,\beta})\hskip -1mm+ \hskip -1mm f(r,X_{r}^{t,x,
\alpha,\beta}\hskip -0.1cm, (\alpha,
\beta)_{r})\nabla\varphi(r,X_{r}^{t,x,\alpha,
\beta})\hskip -0.1cm\right)\hskip -0.1cm dr|\mathcal{F}_{i}\right]\hskip -0.5mm.\\
\end{array}\ee
\noindent Here we have denoted by $(\alpha,\beta)_r$ the unique
couple of control processes $(u,v)\in{\cal U}_{t,t_{l}^{n}}^{
\Pi_{n}}\times {\cal V}_{t,t_{l}^{n}}^{\Pi_{n}}$ at time $r$,
associated with $(\alpha,\beta)\in {\cal A}_{t,t_{l}^{n}}^{\Pi_{n}}
\times {\cal B}_{t,t_{l}^{n}}^{\Pi_{n}}$ by Lemma \ref{lemma 3.1}.
Let us introduce the continuity modulus
$$ m(\delta):=\sup_{|r-t|+|y-x|\le \delta,u\in U,v\in
V}\left|\left((\frac{\partial}{\partial
r}\varphi)(r,y)+f(r,y,u,v)\nabla
\varphi(r,y)\right)-\left((\frac{\partial}{\partial
r}\varphi(t,x)+f(t,x,u,v)\nabla\varphi(t,x)\right)\right|,$$

\smallskip

\noindent $\delta>0.$ Recalling that the function $f(.,.,u,v)$ is
bounded and uniformly continuous, uniformly with respect to
$(u,v)\in U\times V$, and that the first order derivatives of
$\varphi$ are bounded continuous functions, we see that
$m:R_+\rightarrow R_+$ is an increasing function with
$m(\delta)\rightarrow 0$, as $\delta \downarrow 0.$ Thus, taking
into account that $|X_{r}^{t,x, \alpha,\beta}-x|\le C_f|r-t|\le
C_f|t_l^n-t|,\ r\in[t,t_l^n],$ we obtain
\be\label{3.36}\begin{array}{rcl} &
&\displaystyle\left|(\frac{\partial}{\partial r}\varphi)(r,
X_{r}^{t,x,u,v})+f(r,X_{r}^{t,x,u,v},u_{r},v_{r})\nabla\varphi(r,
X_{r}^{t,x,u,v}))-(\frac{\partial}{\partial t}\varphi(t,x)+f(t,x,
u_{r},v_{r})\nabla\varphi(t,x))\right|\\
& &\leq m(C|t_{l}^{n}-t|),\quad r\in[t,t_l^n].
\end{array}\ee
\noindent (The constant $C$ depends on $x$, fixed in this proof.)
Consequently, thanks to (\ref{3.35}),

\be\label{3.37}\begin{array}{rcl} & &\displaystyle
-3\rho-(t_{l}^{n}-t)\left(\frac{\partial}{\partial t}\varphi(t,x)+
m(C|t_{l}^{n}-t|)\right)\\
& &\leq\displaystyle \esssup_{\alpha\in\mathcal
{A}_{t,t_{l}^{n}}^{\Pi_{n}}} \essinf_{\beta\in\mathcal{B}_{t,
t_{l}^{n}}^{\Pi_{n}}}E\left[\int_{t}^{t_{l}^{n}}f(t,x,(\alpha,
\beta)_{r})\nabla\varphi(t,x)dr|\mathcal {F}_{i}\right],\,
P\mbox{-a.s.}
\end{array}\ee

\noindent Similarly to the argument of (\ref{3.12}) in Step 1 of the proof of Lemma
\ref{lemma 1a} we can show there
is an NAD strategy $\alpha^\rho\in\mathcal{A}_{t,
t_{l}^{n}}^{\Pi_{n}}$ such that
\be\label{3.38bis}\begin{array}{rcl} & &\displaystyle
-4\rho-(t_{l}^{n}-t)\left(\frac{\partial}{\partial t}\varphi(t,x)+
m(C|t_{l}^{n}-t|)\right)\\
& &\leq\displaystyle \essinf_{\beta\in\mathcal{B}_{t,
t_{l}^{n}}^{\Pi_{n}}}E\left[\int_{t}^{t_{l}^{n}}f(t,x,(\alpha^\rho,
\beta)_{r})\nabla\varphi(t,x)dr|\mathcal {F}_{i}\right].
\end{array}\ee

\noindent Thus, since $\mathcal {V}^{\Pi_n}_{t,t_{l}^{n}}\subset
\mathcal {B}^{\Pi_n}_{t,t_{l}^{n}}$ (Indeed, the controls
$v\in\mathcal {V}^{\Pi_n}_{t,t_{l}^{n}}$ are identified with
$\beta^v\in \mathcal {B}^{\Pi_n}_{t,t_{l}^{n}}$, where
$\beta^v(u):=v,\, u\in \mathcal {U}^{\Pi_n}_{t,t_{l}^{n}}$.), we
obtain from (\ref{3.38bis}) that, for all $v\in\mathcal
{V}^{\Pi_n}_{t,t_{l}^{n}},$

\be\label{3.38}\displaystyle
-4\rho-(t_{l}^{n}-t)\left(\frac{\partial}{\partial t}\varphi(t,x)+
m(C|t_{l}^{n}-t|)\right) \leq\displaystyle
E\left[\int_{t}^{t_{l}^{n}}f(t,x,(\alpha^\rho,
v)_{r})\nabla\varphi(t,x)dr|\mathcal {F}_{i}\right].\ee

\noindent Let $v\in\mathcal {V}^{\Pi_n}_{t,t_{l}^n}$ be now of the
special form $v:=\sum_{j=i+1}^{l}\xi_{j}I_{[t\vee t_{j-1}^{n},t\vee
t_{j}^{n})},\ \ \xi_{j}\in L^0(\Omega,\mathcal {H}_{j},P;V).$
\noindent Then, \be\label{3.39}\begin{array}{rcl} \displaystyle
E\left[\int_t^{t_l^n}f(t,x,(\alpha^\rho,v)_r)\nabla\varphi(t,x){\rm
d}r|\mathcal{F}_i\right] =\sum\limits_{j=i+1}^lE\left[ \int_{t\vee
t_{j-1}^n}^{t\vee
t_j^n}f(t,x,\alpha^\rho(v)_r,\xi_j)\nabla\varphi(t,x){\rm
d}r|\mathcal{F}_i\right].\end{array}\ee

\noindent Let us put $u^\rho:=\alpha^\rho(v)\in\mathcal
{U}^{\Pi_n}_{t,t_{l}^{n}},$ and let $i+1\le j\le l.$ Then, due to
the definition of the controls from $\mathcal {U}^{\Pi_n}_{t,
t_{l}^{n}}$, there exist an partition $(\Gamma_k)_{k\ge 1}\subset
{\cal F}_{j-1}$ of $\Omega$ and a sequence $(u^{k})_{k\ge 1} \subset
L^0_{{\cal G}_j}(t\vee t_{j-1}^{n},t\vee t_{j}^{n};U)$ such that,
for the restriction of $u^\rho$ to $[t\vee t_{j-1}^{n},t\vee
t_{j}^{n})$,

\centerline{$\displaystyle u^\rho_{|[t\vee t_{j-1}^{n},t\vee
t_{j}^{n})}= \sum_{k\ge 1}I_{\Gamma_k}u^{k}.$}

\noindent Consequently, recalling that $\xi_j\in L^0(\Omega,\mathcal
{H}_{j},P;V)$ and that the three $\sigma$-fields ${\mathcal
{G}_{j}},{\mathcal {H}_{j}}$ and ${\mathcal {F}_{j-1}}$ are mutually
independent, we have

\be\label{3.41}\begin{array}{lll} & & \displaystyle E\left[\int_{
t\vee t_{j-1}^{n}}^{t\vee
t_j^{n}}f(t,x,\alpha^\rho(v)_r,\xi_j)\nabla\varphi(t,x){\rm
d}r|\mathcal{F}_i\right]\\
&=&\displaystyle E\left[\sum_{k\geq1}I_{\Gamma_k}\int_{t\vee
t_{j-1}^n}^{t \vee t_j^n}
E\left[f(t,x,u_r^{k},\xi_j)\nabla\varphi(t,x)|
\mathcal{F}_{j-1}\right]{\rm d}r |\mathcal{F}_i\right]\\
&=&\displaystyle E\left [\sum_{k\geq1}I_{\Gamma_k}\int_{t\vee
t_{j-1}^{n}}^{t\vee t_j^{n}}\left(\int_{U\times
V}f(t,x,u,v)\nabla\varphi(t,x)P_{u_r^{k}}({\rm d}u)\otimes
P_{\xi_j}({\rm d}v)\right){\rm d}r|
\mathcal{F}_i\right]\\
&\leq&\displaystyle (t\vee t_j^{n}-t\vee
t_{j-1}^{n})\cdot\sup_{\mu\in\triangle U}\left(\int_{U\times
V}f(t,x,u,v)\nabla \varphi(t,x)\mu({\rm d}u)\otimes P_{\xi_j}({\rm
d}v)\right).
\end{array}
\ee

\noindent Recall that $\widetilde{f}(t,x,\mu,\nu):=\displaystyle
\int_{U\times V}f(t,x,u,v)\mu({\rm d}u)\nu({\rm d}v).$ Hence, from
(\ref{3.38}), (\ref{3.39}) and (\ref{3.41}),

\be\label{3.43bis}\begin{array}{lll} & &\displaystyle
-4\rho-(t_{l}^{n}-t)\left(\frac{\partial}{\partial t}\varphi(t,x)+
m(C|t_{l}^{n}-t|)\right)\\
& \leq &\displaystyle E\left[\int_{t}^{t_{l}^{n}}f(t,x,
(\alpha^\rho, v)_{r})\nabla\varphi(t,x)dr| \mathcal {F}_{i}\right]\\
&\leq &\displaystyle \sum_{j=i+1}^l(t\vee t_j^n-t\vee
t_{j-1}^{n})\cdot \sup_{\mu\in\Delta U}\widetilde{f}(t,x,
\mu,P_{\xi_j})\nabla\varphi(t,x),
\end{array}\ee

\noindent and from the arbitrariness of the random variables
$\xi_{j}\in L^0(\Omega,\mathcal {H}_{j},P;V),\, i+1\le j\le l$ and
the fact that $\Delta V=\{P_\xi\, |\, \xi\in L^0(\Omega,\mathcal
{H}_{j},P;V)$, we conclude

\smallskip

\be\label{3.43}\begin{array}{lll} & &\displaystyle
-4\rho-(t_{l}^{n}-t)\left(\frac{\partial}{\partial t}\varphi(t,x)+
m(C|t_{l}^{n}-t|)\right)\\
&\leq &\displaystyle \sum_{j=i+1}^l(t\vee t_j^{n}-t\vee
t_{j-1}^{n})\cdot \inf_{\nu\in\Delta V}\sup_{\mu\in
\Delta U}\widetilde{f}(t,x, \mu,\nu)\nabla\varphi(t,x)\\
&= &\displaystyle (t_l^{n}-t)\cdot \inf_{\nu\in\Delta V}\sup_{\mu\in
\Delta U} \widetilde{f}(t,x,\mu,\nu)\nabla\varphi(t,x).
\end{array}\ee

\smallskip

\noindent We choose now $\varepsilon>0$ arbitrarily small and we put
 $\rho=\varepsilon^2.$ For $n\ge n_\rho$\ large enough we can
find some $l\, (i+1\le l\le n)$, such that $\varepsilon \leq
t_l^{(n)}-t\leq 2\varepsilon.$ Indeed, recall that the mesh of
$\Pi_n$ converges to zero, as $n\rightarrow+\infty.$ Then it follows
from (\ref{3.43}) that

\be\label{3.44bis}\begin{array}{lll} & &\displaystyle
-4(t_l^{n}-t)^2-(t_l^{n}-t)\left(\frac{\partial}{\partial
t}\varphi(t,x)+
m(C|t_{l}^{n}-t|)\right)\\
&\leq &\displaystyle (t_l^{n}-t)\cdot \inf_{\nu\in\Delta
V}\sup_{\mu\in \Delta U}
\widetilde{f}(t,x,\mu,\nu)\nabla\varphi(t,x).
\end{array}\ee

\noindent Consequently, dividing both sides of this latter relation
by $t_l^{n}-t $ and taking the limit as $\varepsilon\rightarrow 0$,
we obtain

\be\label{3.44ter}\displaystyle\frac{\partial}{\partial
t}\varphi(t,x)+\inf_{\nu\in \Delta  V}\sup_{\mu\in \Delta
 U}\widetilde{f}(t,x,\mu,\nu)\nabla\varphi(t,x)\geq0.\ee

\smallskip

\noindent In order to conclude, we remark that, for all $(t,x,p)\in
[0,T]\times R^d\times R^d,$ the function
$H(t,x,\mu,\nu,p):=\widetilde{f}(t,x,\mu,\nu)p$ is bilinear in
$(\mu,\nu)\in \Delta U\times \Delta V$. The spaces $\Delta U$ and
$\Delta V$ are compact and convex. Consequently,
\be\label{3.45bis}\displaystyle \inf_{\nu\in\Delta V}\sup_{\mu\in
\Delta U} \widetilde{f}(t,x,\mu,\nu)p=\sup_{\mu\in \Delta
U}\inf_{\nu\in\Delta V} \widetilde{f} (t,x,\mu,\nu)p,\ (t,x,p)\in
[0,T]\times R^d\times R^d,\ee

\noindent and relation (\ref{3.44}) follows from (\ref{3.44ter}).
The proof is complete.
\end{proof}
In order to complete the proof of Proposition \ref{th3.3} we also
have to prove the following
\begin{lemma}\label{lemma2-viscosity V}
The function $V$ is the viscosity supersolution of the
Hamilton-Jacobi-Isaacs equation (\ref{3.33}).
\end{lemma}
\begin{proof} In the proof of Lemma \ref{lemma1-viscosity V} we have
already noticed that $V(T,x)=g(x),\, x\in R^d.$ Let us fix again
$(t,x)\in [0,T)\times R^d$ and consider a test function $\varphi\in
C^1_b([0,T]\times\mathbb{ R}^d)$ which is bounded together with its
first order derivatives, such that $ V(t,x)-\varphi(t,x)=0\leq
V-\varphi$ on $[0,T]\times R^d$. In order to prove the statement we
have to show that \be\label{3.45}\displaystyle
\frac{\partial}{\partial t}\varphi(t,x)+\sup_{\mu\in \Delta
U}\inf_{\nu\in \Delta
 V}\widetilde{f}(t,x,\mu,\nu)\nabla\varphi(t,x)\leq0.\ee

\noindent Let us suppose that this latter relation doesn't hold
true. Then, there exist $\delta>0,$\ and $\mu^*\in\Delta U$\ such
that

\be\label{3.46}\begin{array}{lll} 0<\delta&<&\displaystyle
\frac{\partial}{\partial t}\varphi(t,x)+\sup_{\mu\in \Delta
U}\inf_{\nu\in \Delta
 V}\widetilde{f}(t, x,\mu,\nu)\nabla\varphi(t,x)\\
&=&\displaystyle
 \frac{\partial}{\partial
t}\varphi(t,x)+\inf_{\nu\in \Delta
 V}\widetilde{f}(t, x,\mu^*,\nu)\nabla\varphi(t,x) \\
&\leq& \displaystyle \frac{\partial}{\partial
t}\varphi(t,x)+\widetilde{f}(t, x,\mu^*,\nu)\nabla\varphi(t,x),
\end{array}
\ee

\noindent for all $\nu\in\Delta  V.$ On the other hand, given an
arbitrarily small $\rho>0$\ and $M\ge C\rho^{-1}(|x|+C_fT)$, there
exists $n_{\rho}\ge 1,$ such that for all $n\geq n_{\rho}$,
$$|\varphi(t,x)-V^{\Pi_n}(t,x)|\leq\rho,\quad
V^{\Pi_n}(s,y)\geq \varphi(s,y)-\rho,\, s\in[0,T],\, |y|\le M.$$
\noindent Let $n\geq n_{\rho}$. Adapting the argument of the
proof of the preceding Lemma 4.2 and using the notations introduced
there, we first deduce from the DPP (Theorem \ref{th3.2}) with respect to the partition $\Pi_n$ that
\be\label{3.34a}\begin{array}{lll}
\varphi(t,x)+\rho&\geq& V^{\Pi_{n}}(t,x)\\
&=&\esssup_{\alpha\in\mathcal
{A}_{t,t_{l}^{n}}^{\Pi_{n}}}\essinf_{\beta\in\mathcal
{B}_{t,t_{l}^{n}}^{\Pi_{n}}}E[V^{\Pi_{n}}(t_{l}^{n},X_{t_{l}^{n}
}^{t,x,\alpha,\beta})|\mathcal{F}_{i}]\\
&\geq &\esssup_{\alpha\in\mathcal{A}_{t,t_{l}^{n}}^{\Pi_{n}}}
\essinf_{\beta\in\mathcal{B}_{t,t_{l}^{n}}^{\Pi_{n}}}
E[\varphi(t_{l}^{n},X_{t_{l}^{n}}^{t,x,\alpha,\beta})|\mathcal
{F}_{i}]\\
& & \hskip 3cm -C P\{\, |X_{t_{l}^{n}}^{t,x,\alpha,\beta}
|>M|\mathcal{F}_{i}\}-\rho\\
&\geq &\esssup_{\alpha\in\mathcal{A}_{t,t_{l}^{n}}^{\Pi_{n}}}
\essinf_{\beta\in\mathcal{B}_{t,t_{l}^{n}}^{\Pi_{n}}}
E[\varphi(t_{l}^{n},X_{t_{l}^{n}}^{t,x,\alpha,\beta})|\mathcal
{F}_{i}]-2\rho,\ P\mbox{-a.s.}
\end{array}\ee

\noindent Consequently,
\be\label{3.35a}\hskip -0.4cm\begin{array}{rl} \displaystyle  3\rho\geq
 \displaystyle \esssup_{\alpha\in\mathcal {A}_{t,t_{l}^{n}}^{\Pi_{n}}} \essinf_{\beta\in\mathcal {B}_{t,t_{l}^{n}}^{\Pi_{n}}}
E\left[ \int_{t}^{t_{l}^{n}}\hskip
-2mm\left(\frac{\partial}{\partial r}\varphi(r,
X_{r}^{t,x,\alpha,\beta})\hskip -1mm+ \hskip -1mm f(r,X_{r}^{t,x,
\alpha,\beta},(\alpha, \beta)_{r})\nabla\varphi(r,X_{r}^{t,x,
\alpha, \beta})\right)dr|\mathcal{F}_{i}\right]\hskip -0.5mm,\\
\end{array}\ee

\noindent and using the continuity modulus $m(.)$ introduced in the
proof of Lemma \ref{lemma1-viscosity V} we obtain
\be\label{3.37a}\begin{array}{rcl} & &\displaystyle
3\rho-(t_{l}^{n}-t)\left(\frac{\partial}{\partial t}\varphi(t,x)-
m(C|t_{l}^{n}-t|)\right)\\
& &\geq\displaystyle \esssup_{\alpha\in\mathcal
{A}_{t,t_{l}^{n}}^{\Pi_{n}}} \essinf_{\beta\in\mathcal{B}_{t,
t_{l}^{n}}^{\Pi_{n}}}E\left[\int_{t}^{t_{l}^{n}}f(t,x,(\alpha,
\beta)_{r})\nabla\varphi(t,x)dr|\mathcal {F}_{i}\right],\,
P\mbox{-a.s.}
\end{array}\ee

\noindent In the next step, observing that we can identify
$\mathcal{U}_{t,t_l^{n}}^{\Pi_n}$ as a subset of
$\mathcal{A}_{t,t_l^{n}}^{\Pi_n}$, and choosing
$u\in\mathcal{U}_{t,t_l^{n}}^{\Pi_n}$ of the form
$u=\sum_{j=i+1}^l\xi_jI_{[t\vee t_{j-1}^{n},t\vee t_j^{n})},$ with
$\xi_j\in {L}^0(\Omega,{\mathcal{G}_j},P; U)$ such that
$P_{\xi_j}=\mu^*\ (i+1\le j\le l)$, we get
\be\label{3.37b}\begin{array}{rcl} & &\displaystyle
3\rho-(t_{l}^{n}-t)\left(\frac{\partial}{\partial t}\varphi(t,x)-
m(C|t_{l}^{n}-t|)\right)\\
&\geq &\displaystyle\essinf_{\beta\in\mathcal{B}_{t,
t_{l}^{n}}^{\Pi_{n}}}E\left[\int_{t}^{t_{l}^{n}}f(t,x,(u,
\beta(u)_r))\nabla\varphi(t,x)dr|\mathcal {F}_{i}\right],\,
P\mbox{-a.s.}
\end{array}\ee

\noindent In analogy to the argument of (\ref{3.12}) in Step 1 of the proof of Lemma
\ref{lemma 1a} we now can construct some
$\beta^\rho\in\mathcal{B}_{t,t_l^n}^\Pi$ (depending on the control
process $u$) such that

\be\label{3.37c}\begin{array}{rcl} & &\displaystyle
4\rho-(t_{l}^{n}-t)\left(\frac{\partial}{\partial t}\varphi(t,x)-
m(C|t_{l}^{n}-t|)\right)\\
&\geq&\displaystyle E\left[\int_{t}^{t_{l}^{n}}f(t,x,(u,
\beta^\rho(u)_r))\nabla\varphi(t,x)dr|\mathcal {F}_{i}\right]\\
&\geq&\displaystyle E\left[\sum_{j=i+1}^l\int_{t\vee
t_{i-1}^{n}}^{t\vee t_{j}^{n}}f(t,x,(\xi_j,
\beta^\rho(u)_r))\nabla\varphi(t,x)dr|\mathcal {F}_{i}\right].\\
\end{array}\ee

\noindent We put now $v:=\beta^\rho(u)$, and we observe that, for
any $i+1\le j\le n$,  the restriction of $v$ to the interval $[t\vee
t_{j-1}^{n}, t_j^{n})$ belongs to $\mathcal{V}_{t\vee
t_{j-1}^{n},t_j^{n}}^{\Pi_n}$. Consequently, by definition,
$v|_{[t\vee t_{j-1}^{n},t\vee t_j^{n})}$ is of the form $v|_{[t\vee
t_{j-1}^n,t\vee
t_j^n)}=\displaystyle\sum_{k\geq1}I_{\Gamma_k}v^{k},$ where
$(\Gamma_k)_{k\geq1}\subset \mathcal{F}_{j-1}$ is a partition of
$\Omega$ and $v_{k}\in {L}^0_{\mathcal{H}_j}(t_{j-1}^{n}, t_j^{n};
V),\, k\ge 1$ (For simplicity of notations we have suppressed in this
representation the dependence on $j$). Thus, the independence of the
three $\sigma$-fields $\mathcal{H}_j,\ \mathcal{G}_j$ and
$\mathcal{F}_{j-1}$yields

\be\label{3.52}\begin{array}{lll} &&\displaystyle E\left[\int_{t\vee
t_{j-1}^{n}}^{t\vee t_j^{n}}f(t,x,\xi_j, \beta^\rho(u)_r)\nabla
\varphi(t,x){\rm d}r|\mathcal{F}_i^{n}\right]\\
&=&\displaystyle E\left[\sum_{k\geq 1}I_{\Gamma_k}\int_{t\vee
t_{j-1}^{n}}^{t\vee t_j^{n}} E\left[f(t,x,\xi_j,v^{k}_r)
\nabla\varphi(t,x)|\mathcal{F}_{j-1}\right]
{\rm d}r|\mathcal{F}_i\right] \\
&=&\displaystyle E\left[\sum_{k\geq 1}I_{\Gamma_k}\int_{t\vee
t_{j-1}^n }^{t\vee t_j^{n}}\widetilde{f}(t,x, \mu^*,
P_{v^{k}_r})\nabla\varphi(t,x){\rm d}r|\mathcal{F}_i\right]\\
&\geq& \displaystyle (t\vee t_j^{n}-t\vee t_{j-1}^{n})\cdot
E\left[\sum_{k\geq1}I_{\Gamma_k}\inf_{\nu\in \Delta
V}(\widetilde{f}(t,x,\mu^*,\nu)\nabla\varphi(t,x))|
\mathcal{F}_i\right]\\
&=&\displaystyle (t\vee t_j^{n}-t\vee t_{j-1}^{n})\cdot\inf_{\nu \in
\Delta  V}\widetilde{f}(x,\mu^*,\nu)\nabla\varphi(t,x)).
\end{array}
\ee

\noindent Therefore, summing up (\ref{3.52}) with respect to $j$ and
substituting the result in (\ref{3.37c}) we obtain

\be\label{3.53}\begin{array}{lll} & &\displaystyle
4\rho+(t_{l}^{n}-t)m(C|t_{l}^{n}-t|)\\
&\geq &\displaystyle (t_l^n-t)\cdot\left(\frac{\partial}{\partial
t}\varphi(t,x)+\inf_{\nu\in \Delta
 V}\widetilde{f}(x,\mu^*,\nu)\nabla\varphi(t,x)\right)\\
&\geq &\delta(t_l^n-t).
\end{array}
\ee

\noindent Let now $\varepsilon>0,\ \rho=\varepsilon^2$ and
$|\Pi_n|>0$\ be small enough, such that $t_l^n$\ can be chosen such
that $\frac{\varepsilon}{2}\leq t_l^n-t\leq \varepsilon$. Then, from
(\ref{3.53}) we have \be\label{3.54}4\varepsilon^2+\varepsilon
m(C\varepsilon)\geq \frac{\varepsilon}{2}\delta.\ee

\noindent Thus, first dividing this latter relation by $\varepsilon$
and after letting $\varepsilon\rightarrow 0$, we get $\delta\leq 0$,
which contradicts $\delta>0$ in (\ref{3.46}). Therefore, our
hypothesis is wrong and we have (\ref{3.45}). The proof is complete.
\end{proof}

In analogy to Proposition \ref{th3.3} we can prove the following
\begin{proposition}
Also the function $U\in C_b([0,T]\times R^d)$ is a viscosity
solution of the Hamilton-Jacobi-Isaacs equation (\ref{3.33}).
\end{proposition}

Finally, we are able to prove Theorem \ref{main result}.
\begin{proof}
Due to relation (\ref{3.45bis}) we know that the bounded continuous
functions $V$ and $U$ are viscosity solutions of the same
Hamilton-Jacobi-Isaacs equation. On the other hand, since the
Hamiltonian of this equation
$$\displaystyle H(t,x,p)=\inf_{\nu\in \Delta
 V}\sup_{\mu\in \Delta
 U}(\widetilde{f}(t,x,\mu,\nu)p),\, (t,x,p)\in  [0,T]\times
R^d\times R^d,$$

\noindent is bounded and continuous, Lipschitz in $z$, uniformly
with respect to $(t,x)\in[0,T]\in R^d$, and

$$|H(t,x,p)-H(t,x',p)|\le C|p||x-x'|,\ \ x,x'\in R^d,
(t,p)\in[0,T]\times R^d,$$

\noindent it is by now well-known, that the viscosity solution of
the Hamilton-Jacobi-Isaacs equation (\ref{3.33}) is unique in the
class of continuous functions with at most polynomial growth.
Consequently, $V=U$. On the other hand, recall that we have got $V$
and $U$ as limit over a converging subsequence of the sequence
$V^{\Pi_n}$ and $U^{\Pi_n}$, respectively, where $(\Pi_n)_{n\ge 1}$
is an arbitrarily chosen sequence of partitions of $[0,T]$ such that
$|\Pi_n|\rightarrow 0\ (n\rightarrow +\infty)$. Therefore, since the
limit of the converging subsequence doesn't depend on the choice of
the sequence, it follows that $V^\Pi$ and $U^\Pi$ converge along all
sequence of partitions $\Pi$ with $|\Pi|\rightarrow 0$, and the
limit is $V=U.$ The proof of Theorem \ref{main result} is complete.
\end{proof}

\section*{Acknowledgments}  The work of Rainer Buckdahn and Marc Quincampoix have been partially supported by the Commission of the European Communities under the 7-th Framework Programme Marie Curie Initial Training Networks
Project ``Deterministic and Stochastic Controlled Systems and
Applications" FP7-PEOPLE-2007-1-1-ITN, no. 213841-2 and project
SADCO, FP7-PEOPLE-2010-ITN, No 264735. This was also supported
partially by the French National Research Agency ANR-10-BLAN 0112.

 The work of Juan LI has been supported by the NSF of P.R.China (No. 11071144, 11171187), Shandong Province (No. BS2011SF010), SRF for ROCS (SEM), 111 Project (No. B12023).

\end{document}